\documentclass[12pt,reqno]{amsart}

\usepackage[T1]{fontenc}
\usepackage[utf8]{inputenc}
\usepackage[english]{babel}
\usepackage{amsmath,amssymb,amsthm,mathtools}
\usepackage{dsfont}
\usepackage{microtype}
\usepackage{indentfirst}
\usepackage[left=2.7cm,right=2.7cm,top=2.8cm,bottom=2.8cm]{geometry}

\usepackage[
  colorlinks=true,
  linkcolor=blue,
  citecolor=blue,
  urlcolor=blue
]{hyperref}

\numberwithin{equation}{section}

\usepackage[mathlines,pagewise]{lineno}

\newtheorem{theorem}{Theorem}[section]
\newtheorem{proposition}[theorem]{Proposition}
\newtheorem{lemma}[theorem]{Lemma}

\theoremstyle{definition}
\newtheorem{definition}{Definition}[section]

\theoremstyle{remark}
\newtheorem{remark}[theorem]{Remark}

\newcommand{\dd}{\,\mathrm{d}}
\newcommand{\Tr}{\operatorname{Tr}}
\newcommand{\tr}{\operatorname{tr}}
\newcommand{\diver}{\operatorname{div}}
\newcommand{\R}{\mathds R}
\newcommand{\HH}{\mathcal H}
\newcommand{\BV}{\mathrm BV}
\newcommand{\dist}{\operatorname{dist}}
\newcommand{\supp}{\operatorname{supp}}

\title[Weighted Sobolev Trace Inequalities]
{Sharp Weighted Sobolev Trace Inequalities on Convex Cones
and Rigidity of Extremals}

\author[J. M. do \'{O}]{Jo\~{a}o Marcos do \'{O}}
\address{Universidade Federal da Para\'iba, Departamento de Matem\'atica
\newline\indent
58051-900 Jo\~{a}o Pessoa-PB, Brazil}
\email{\href{mailto:jmbo@mat.ufpb.br}{jmbo@mat.ufpb.br}}

\author[L. M. P. Silva]{Lucas M. P. Silva}
\address{Universidade Federal da Para\'iba, Departamento de Matem\'atica
\newline\indent
58051-900 Jo\~{a}o Pessoa-PB, Brazil}
\email{\href{mailto:lmpes@mat.ufpb.br}{lmpes@mat.ufpb.br}}



\begin{document}


\begin{abstract}
We study critical weighted Sobolev trace inequalities on convex cones
endowed with positively homogeneous weights. 
Using mass transportation, we first establish the continuity of the critical trace embedding under
general homogeneity and concavity assumptions on the weight. 
We then identify a compatibility condition between the weight, the cone, and a
distinguished direction that makes the transport argument sharp. 
In this regime, we determine the optimal constant and prove the existence of
extremals. 
Finally, we obtain a complete classification of all nontrivial extremals modulo translation in the lineality space.
\end{abstract}

\maketitle

\section{Introduction}

Critical Sobolev trace inequalities quantify how the interior energy of a
function controls the size of its boundary trace. In the presence of a
positively homogeneous weight, the Euclidean dimension is naturally
replaced by an effective dimension determined by the degree of homogeneity
of the weight. Convex cones provide a natural geometric setting in which
this homogeneity is compatible with both the domain and the weighted
energy.

The purpose of this paper is to study the corresponding critical weighted
trace problem by means of mass transportation. A basic feature of the
problem is that continuity and sharpness are governed by different levels
of compatibility between the geometry of the cone and the weight. Under
general homogeneous--concave assumptions, the transport method gives a
continuous critical trace embedding with an explicit constant. However, the same
argument is not automatically sharp. A weighted-divergence term
appearing in the boundary integration by parts may produce a genuine loss.

This leads naturally to two different geometric regimes.
We first establish the continuous trace inequality under a
transversality condition. We then introduce a distinct
transport-compatible regime in which the weighted-divergence
loss disappears and the transport argument becomes sharp.
In this regime, we determine the optimal constant, construct
extremals, and obtain a complete classification of the equality
cases, including the translation freedom generated by the
lineality space of the cone.

\subsection{Related results}

Sharp Sobolev trace inequalities have been studied from several different
points of view. In the Euclidean half-space, Lions proved the existence of
extremals for the critical trace quotient by means of the
concentration--compactness principle \cite{Lions1985}. In the quadratic
case, Escobar identified the sharp constant and classified the extremals
using conformal invariance \cite{Escobar1988}, while Beckner obtained
related sharp inequalities using harmonic analysis on the sphere
\cite{Beckner1993}. For general exponents, Nazaret~\cite{Nazaret2006} used mass transportation
to determine the sharp constant in the Sobolev trace inequality on the
half-space and to exhibit a family of extremals, although without a
complete rigidity result. This work showed that optimal transport
provides a natural framework for sharp trace inequalities beyond the
conformally invariant setting. More generally, optimal transport has become an effective tool in the
study of sharp functional inequalities. Cordero-Erausquin, Nazaret, and
Villani~\cite{CorderoNazaretVillani2004} developed a mass-transportation
approach to sharp Sobolev and Gagliardo--Nirenberg inequalities, including
the analysis of equality cases. In the trace setting, Maggi and Neumayer~\cite{MaggiNeumayer2017}
introduced a family of sharp trace--Sobolev inequalities on half-spaces
connecting the classical Sobolev and Escobar inequalities. In particular,
they obtained a complete classification of the corresponding equality
cases, providing the relevant rigidity result in the unweighted
half-space setting.

Related trace
inequalities on convex cones were obtained by Zugmeyer
\cite{Zugmeyer2019}, who proved sharp
Gagliardo--Nirenberg--Sobolev trace inequalities on cones and weighted
trace inequalities on epigraphs of convex functions. More precisely, the cone trace inequality of
Zugmeyer~\cite[Corollary~1.3]{Zugmeyer2019}
uses the projected boundary measure $\mathrm{d}x'$.
For an epigraph cone and $e=-e_N$, one has
$\mathrm{d}x'=(e\cdot\eta)\,\mathrm{d}\sigma$,
where $\eta$ is the outward unit normal.
The present paper instead uses the same homogeneous
density $w$ in the interior measure $w\,\mathrm{d}x$
and the boundary measure $w\,\mathrm{d}\sigma$.

A parallel line of research concerns weighted Sobolev and isoperimetric
inequalities on cones. Cabr\'e and Ros-Oton
\cite{CabreRosOton2013} proved sharp Sobolev and isoperimetric inequalities
for monomial weights, including the best constants and the extremal
functions. This theory was extended by Cabr\'e, Ros-Oton, and Serra
\cite{CabreRosOtonSerra2016} to more general homogeneous weights on convex
cones under a natural concavity assumption. Their results highlight the
role of the effective dimension associated with the homogeneity of the
weight. Using mass transportation, Nguyen \cite{Nguyen2015} established
sharp weighted Sobolev and Gagliardo--Nirenberg inequalities on
half-spaces with monomial weights and described all extremals. 

More general weighted inequalities involving two homogeneous
weights were studied by Balogh, Guti\'errez, and
Krist\'aly~\cite{BaloghGutierrezKristaly2021}.
Balogh, Don, and Krist\'aly~\cite{BDK2024} established sharp
weighted log-Sobolev inequalities on convex cones and
characterized their equality cases. They subsequently
developed weighted Gagliardo--Nirenberg inequalities involving
homogeneous weights, together with equality results under
suitable structural assumptions~\cite{BDK2025}.
These contributions concern interior inequalities.
The present paper addresses a boundary trace problem in which
the same homogeneous density governs both the interior energy
and the boundary measure. In the anisotropic setting,
Ciraolo, Figalli, and Roncoroni \cite{CiraoloFigalliRoncoroni2020}
developed related rigidity results for critical equations in convex cones.

Weighted trace inequalities have also appeared in connection with
fractional and extension problems. In particular, Case \cite{Case2020}
proved sharp weighted Sobolev trace inequalities on the upper half-space
for power weights arising in extension problems for fractional powers of
the Laplacian. These results concern a distinguished class of degenerate
weights and the flat geometry of the half-space.

The aim of the present paper is to develop a weighted trace theory on
convex cones that separates the general continuity problem from the
additional geometric structure required for sharpness and rigidity. We
consider positively homogeneous weights satisfying a concavity condition,
with the same weight governing both the interior energy and the boundary
trace measure.
\subsection{Main results}\label{sec:main-results}
We begin with the setting of the continuous trace inequality.

Let $N\geq3$, and denote by $\mathds{S}^{N-1}
=
\{x\in\mathds{R}^N:|x|=1\}$ the unit sphere in $\mathds{R}^N$. Let
$D\subset\mathds{S}^{N-1}$ be a Lipschitz domain and define
\[
\Sigma
:=
\{x=rq\in\mathds{R}^N:r>0,\ q\in D\}.
\]
Assume that $\Sigma$ is an open convex cone with Lipschitz boundary and
that, after a rotation,
\[
D
\subset
\{q=(q',q_N)\in\mathds{S}^{N-1}:q_N>0\}.
\]
Equivalently,
\[
\Sigma
\subset
\{x=(x',x_N)\in\mathds{R}^N:x_N>0\}.
\]

Let
\[
-e=(0,\ldots,0,1)\in D,
\qquad
e=(0,\ldots,0,-1).
\]
We denote by $\eta$ the exterior unit normal to $\partial\Sigma$, defined
for $\mathcal H^{N-1}$-almost every boundary point, and write
\[
\mathrm{d}\sigma
:=
\mathrm{d}\mathcal H^{N-1}\big|_{\partial\Sigma}.
\]

Assume that
\[
c_\Sigma
:=
\operatorname*{ess\,inf}_{x\in\partial\Sigma\setminus\{0\}}
e\cdot\eta(x)
>0.
\]
This is a uniform transversality condition between the boundary of
$\Sigma$ and the distinguished direction $e$.

Let $w\in C(\overline\Sigma)$ satisfy
\[
w\geq0
\quad\text{on }\overline\Sigma,
\qquad
w>0
\quad\text{in }\Sigma,
\]
and assume that $w$ is positively homogeneous of degree $a>0$,
\[
w(tx)=t^a w(x),
\qquad
x\in\Sigma,\quad t>0,
\]
and that
\[
w^{1/a}
\quad\text{is concave in }\Sigma.
\]

Set
\[
d:=N+a,
\qquad
1<p<d,
\qquad
 p^*:=\frac{pd}{d-p},
 \qquad
p_*:=\frac{p(d-1)}{d-p},
\qquad
p':=\frac{p}{p-1}.
\]

We denote by $C_c^\infty(\overline\Sigma)$ the restrictions to
$\overline\Sigma$ of functions in $C_c^\infty(\mathds{R}^N)$. The
homogeneous weighted Sobolev space $\dot W^{1,p}(\Sigma,w)$ is defined as
the completion of $C_c^\infty(\overline\Sigma)$ with respect to the
homogeneous norm
\[
\|u\|_{\dot W^{1,p}(\Sigma,w)}
:=
\left(
\int_\Sigma
|\nabla u|^p w\,\mathrm{d}x
\right)^{1/p}.
\]

Our first result gives the continuity of the critical trace operator.

\begin{theorem}[Continuous weighted Sobolev trace embedding]
\label{thm:trace-continuous}
Under the assumptions above, the trace operator, initially defined on
$C_c^\infty(\overline\Sigma)$, extends uniquely to a continuous linear map
\[
\operatorname{Tr}:
\dot W^{1,p}(\Sigma,w)
\longrightarrow
L^{p_*}(\partial\Sigma,w\,\mathrm{d}\sigma).
\]
More precisely,
\begin{equation}
\label{eq:trace-continuous}
\left(
\int_{\partial\Sigma}
|\operatorname{Tr}u|^{p_*}w\,\mathrm{d}\sigma
\right)^{1/p_*}
\leq
C_{\mathrm{MT}}
\left(
\int_\Sigma
|\nabla u|^p w\,\mathrm{d}x
\right)^{1/p}
\end{equation}
for every $u\in\dot W^{1,p}(\Sigma,w)$, where
\begin{equation}
\label{eq:CMT}
C_{\mathrm{MT}}
:=
\left[
c_\Sigma
\left(
\frac{d-p}{p-1}
\right)^{p_*-1}
\left(
\int_\Sigma
|x-e|^{-(d-1)p'}w(x)\,\mathrm{d}x
\right)^{\frac{p-1}{d-p}}
\right]^{-1/p_*}.
\end{equation}
Moreover,
\[
0<
\int_\Sigma
|x-e|^{-(d-1)p'}w(x)\,\mathrm{d}x
<\infty.
\]
\end{theorem}

The estimate in Theorem~\ref{thm:trace-continuous} is not automatically
sharp. The reason is that the weighted Gauss--Green formula produces an
additional term involving the derivative of the weight in the
distinguished direction. Under the assumptions above, this term has a
strict sign, and the boundary step of the mass-transportation argument
generally loses information. This leads us to introduce a second regime
in which the weight and the boundary geometry are compatible with
equality.

We retain the assumptions that $N\geq3$, that $\Sigma$ is an open convex
cone with Lipschitz boundary, and that $w$ is continuous and nonnegative
on $\overline\Sigma$, positive in $\Sigma$, homogeneous of
degree $a>0$, with $w^{1/a}$ concave. We also retain $1<p<N+a$ and the
definitions of the energy space and critical exponents. We use the notation $c_{\Sigma,e}$ for the boundary constant in the sharp
regime in order to distinguish it from the constant $c_\Sigma$ appearing
in Theorem~\ref{thm:trace-continuous}.
The preceding
geometric normalisation and transversality assumptions are replaced by
the following compatibility conditions.

\begin{definition}[Transport-compatible triple]
\label{def:transport-compatible}
We say that $(\Sigma,w,e)$ is \emph{transport-compatible} if
\begin{align}\label{eq:first-comp-cond}
    \operatorname{dist}(e,\overline\Sigma)>0, \qquad e \cdot \nabla w=0 \text{ a.e. in }\Sigma,
\end{align}
the second condition is equivalent to $\partial_{x_N}w=0$
 for the present choice $e=-e_N$. In addition, the weighted boundary measure $w\,\mathrm{d}\sigma$ is nonzero on
$\partial\Sigma$, and there exists a constant $c_{\Sigma,e}>0$ such that

\begin{equation}\label{eq:second-comp-cond}
  e\cdot\eta=c_{\Sigma,e}
\qquad
w\,\mathrm{d}\sigma\text{-almost everywhere on }\partial\Sigma.  
\end{equation}
\end{definition}

The compatibility condition
\[
e \cdot \nabla w=0 \text{ a.e. in }\Sigma
\]
removes the weighted-divergence term
that is responsible for the loss in the general transport argument. The
compatibility condition~\eqref{eq:second-comp-cond}  is the corresponding equality condition for the boundary
geometry. We emphasize, however, that the cone boundary contribution $\Tr T \cdot \eta$, where $T$ is the transport map, is nonpositive and does not vanish merely by compatibility. For the extremizing source/target choice, the transport $T$ is the identity and the remaining cone boundary contribution also vanishes because $x \cdot \eta =0$.  Hence, equality holds in the boundary estimate for this choice.

For a transport-compatible triple $(\Sigma,w,e)$, set
\[
I_{\Sigma,w,e}
:=
\int_\Sigma
|x-e|^{-(d-1)p'}w(x)\,\mathrm{d}x.
\]

Our second result identifies the optimal constant and gives an explicit
family of extremals.

\begin{theorem}[Sharp weighted Sobolev trace inequality and attainment]
\label{thm:sharp-trace}
Assume that $(\Sigma,w,e)$ is transport-compatible. Then
\[
0<I_{\Sigma,w,e}<\infty.
\]
Define
\[
C_{\sharp}
:=
\left[
c_{\Sigma,e}
\left(
\frac{d-p}{p-1}
\right)^{p_*-1}
I_{\Sigma,w,e}^{\frac{p-1}{d-p}}
\right]^{-1/p_*}.
\]
Then
\begin{equation}
\label{eq:sharp-trace}
\left(
\int_{\partial\Sigma}
|\operatorname{Tr}u|^{p_*}w\,\mathrm{d}\sigma
\right)^{1/p_*}
\leq
C_{\sharp}
\left(
\int_\Sigma
|\nabla u|^p w\,\mathrm{d}x
\right)^{1/p}
\end{equation}
for every $u\in\dot W^{1,p}(\Sigma,w)$.

The constant $C_{\sharp}$ is optimal and is attained. More precisely,
equality in \eqref{eq:sharp-trace} is attained by every function of the
form
\begin{equation}
\label{eq:explicit-extremals}
u_{C,\lambda}(x)
=
C\lambda^{\frac{d-p}{p}}
|\lambda x-e|^{-\frac{d-p}{p-1}},
\qquad
C\neq0,
\quad
\lambda>0.
\end{equation}
\end{theorem}

We call a nonzero function $u\in\dot W^{1,p}(\Sigma,w)$
an \emph{extremal} if equality holds in \eqref{eq:sharp-trace}. For the
rigidity argument we shall use the positivity of nonnegative extremals.
This follows from the interior Euler--Lagrange equation and the local
Harnack inequality for equations of $p$-Laplacian type; the details are
given in Section~\ref{sec:preliminaries}.

\bigskip

For the rigidity classification, we make explicit the lineality space
associated with the open convex cone $\Sigma$. Since $\Sigma$ is open,
we define
\[
\operatorname{L}(\Sigma)
:=
\overline{\Sigma}\cap(-\overline{\Sigma}).
\]

The lineality space determines the possible translation symmetries in
the classification of extremals. In particular, translations along
$\operatorname{L}(\Sigma)$ preserve the cone. More precisely, if
$z\in\operatorname{L}(\Sigma)$, then
\[
 \Sigma+z=\Sigma.
\]
Therefore,

\[
x+t z\in\Sigma
\qquad\text{for every } x \in \Sigma,\,t\in\mathbb R.
\]

Under the standing positivity and concavity assumptions, the weight is
automatically invariant along these directions. Indeed, if
$z\in\operatorname{L}(\Sigma)$ and $x\in\Sigma$, then the function
\[
t\longmapsto w(x+t z)^{1/a}
\]
is positive and concave on $\mathbb R$. Hence it must be constant.
Consequently,
\[
w(x+t z)=w(x)
\qquad
\text{for every }t\in\mathbb R,
\]
and therefore
\[
\nabla w(x)\cdot z=0
\]
for almost every $x\in\Sigma$.

We can now state the rigidity result.

\begin{theorem}[Complete classification of extremals]\label{thm:classification}
Assume the hypotheses of Theorem~\ref{thm:sharp-trace}. Then a nonzero function
$u\in\dot W^{1,p}(\Sigma,w)$ is an extremal of \eqref{eq:sharp-trace}
if and only if
\[
u(x)
=
C\lambda^{\frac{d-p}{p}}
\left|\lambda(x-x_0)-e\right|^{-\frac{d-p}{p-1}}
\]
for some
\[
C\neq0,\qquad
\lambda>0,\qquad
x_0\in\operatorname{L}(\Sigma).
\]
\end{theorem}

\begin{remark}\label{rem:vector-e}
The transport-compatible regime is separate from the transversality
regime and is based on a different set of geometric assumptions.
Thus, the assumptions used in the transversality setting are not imposed
here; in particular, we do not require $-e\in D$.
\end{remark}

\subsection{Examples and scope of the main results}

The transport-compatible regime includes both monomial and
nonmonomial weights. To illustrate this, write
$x=(z,y,t)\in\mathbb R^L\times\mathbb R^m\times\mathbb R$,
where $L+m+1=N$, and consider
\[
\Sigma=\mathbb R^L\times(0,\infty)^m\times(0,\infty),
\qquad e=-e_N.
\]
Two admissible choices are
\begin{equation}\label{eq:first-example}
w(z,y,t)=\prod_{i=1}^m y_i^{\alpha_i},
\qquad \alpha_i>0,\quad \sum_{i=1}^m\alpha_i=a,    
\end{equation}
and, for $m\geq2$,
\begin{equation}\label{eq:sec-example}
w(z,y,t)=\bigl(\min\{y_1,\ldots,y_m\}\bigr)^a.
\end{equation}
In both cases, $w^{1/a}$ is concave, $w$ is independent
of $t$, and $w$ vanishes on the lateral coordinate faces.
The weighted boundary measure is therefore concentrated
on the base $t=0$, where $e\cdot\eta=1$. Since
$\operatorname{dist}(e,\Sigma)=1$, these triples are
transport-compatible. The second example~\eqref{eq:sec-example} shows that the
sharp results apply to nonmonomial weights that need not
be differentiable throughout the interior. Moreover,
\[
\operatorname{L}(\Sigma)
=\mathbb R^L\times\{0\}^{m+1},
\]
so the translation parameters in the classification
correspond precisely to the Euclidean factor; they
disappear when $L=0$.

The two geometric regimes considered in this paper are mutually exclusive for a fixed triple under the standing assumptions $a>0$ and
$w>0$ in $\Sigma$.

They share the same underlying transport framework, but impose different
geometric conditions on the interaction between the cone, the weight, and
the distinguished direction. In the transverse regime, the assumptions
imply, almost everywhere,
\[
e\cdot\nabla w(x)
\leq
-a\,w(-e)^{1/a}w(x)^{1-1/a}<0,
\]
whereas transport compatibility requires
\[
\frac{\partial w}{\partial x_N}=0.
\]
Thus, for a given triple $(\Sigma,w,e)$, these two conditions cannot
simultaneously hold. For example, if
\[
\Sigma_b=\{(x',t):t>b|x'|\},\qquad
w(x',t)=t^a,\qquad e=-e_N,\qquad b>0,
\]
then $e\cdot\eta=(1+b^2)^{-1/2}>0$ at regular boundary
points, whereas
\[
e\cdot\nabla w=-a t^{a-1}<0.
\]
Thus Theorem~\ref{thm:trace-continuous} applies, but the compatibility condition
required for the sharp results fails.

These examples clarify the scope of the paper: the same
homogeneous density governs the interior energy and the
boundary trace, while compatibility determines the regime
in which we identify the best constant and classify all
extremals. This boundary problem complements the weighted
interior inequalities discussed above and includes
nonmonomial weights and cones with nontrivial lineality.

\bigskip

\paragraph{Sharpness with a nonmonomial weight:}
An explicit application of Theorem~\ref{thm:sharp-trace} is obtained by taking
\[
\Sigma=(0,\infty)^N,\qquad e=-e_N,\qquad
w(x)=\bigl(\min\{x_1,\ldots,x_{N-1}\}\bigr)^a.
\]
The function $w^{1/a}$ is concave and homogeneous of degree
one, and $w$ is independent of $x_N$. Moreover, the weight
vanishes on every lateral coordinate face. Thus the weighted
boundary measure is concentrated on
\[
\Gamma=(0,\infty)^{N-1}\times\{0\},
\]
where $e\cdot\eta=1$. Since
$\operatorname{dist}(e,\Sigma)=1$, the triple is
transport-compatible, with $c_{\Sigma,e}=1$.
Theorem~\ref{thm:sharp-trace} therefore yields the sharp trace inequality,
attained by
\[
u(x',x_N)=C\lambda^{(d-p)/p}
\left(
\lambda^2|x'|^2+(\lambda x_N+1)^2
\right)^{-\frac{d-p}{2(p-1)}},
\qquad C\neq0,\quad \lambda>0,
\]
where $d=N+a$. In particular, sharpness does not require
a monomial weight or differentiability of the weight
throughout the interior.

\bigskip

\paragraph{Rigidity on cones with lineality:}
To illustrate the translation parameters in Theorem~\ref{thm:classification},
let $0\leq L\leq N-2$, set $m=N-L-1$, and consider
\[
\Sigma=\mathbb R^L\times(0,\infty)^m\times(0,\infty),
\qquad
w(z,y,t)=\prod_{i=1}^m y_i^{\alpha_i},
\qquad
e=-e_N,
\]
where $\alpha_i>0$ and $\sum_i\alpha_i=a$.
The weighted geometric mean $w^{1/a}$ is concave, and
the same boundary argument gives transport compatibility.
Furthermore,
\[
\operatorname{L}(\Sigma)
=\mathbb R^L\times\{0\}^{m+1},
\]
and $w$ is invariant under translations in the $z$ variables.
Theorem~\ref{thm:classification} states that all nonzero extremals are precisely
\[
u(z,y,t)=C\lambda^{(d-p)/p}
\left(
\lambda^2|z-z_0|^2+\lambda^2|y|^2+(\lambda t+1)^2
\right)^{-\frac{d-p}{2(p-1)}},
\]
with $C\neq0$, $\lambda>0$, and $z_0\in\mathbb R^L$.
Thus the translation freedom is determined exactly by
the lineality of the cone. In the pointed case $L=0$,
only amplitude and dilation remain.

\subsection{General strategy of the proofs}

The proofs are built around the mass-transportation method, but transport
plays a different role in the continuous, sharp, and rigidity parts of
the paper. We briefly describe the main ideas.

To prove the continuous trace inequality, we associate with a nonnegative
function $u$ a probability measure in the cone and transport it, through
the gradient of a convex transport potential, onto a suitable target
measure. The Monge--Amp\`ere equation and the concavity of the homogeneous
weight allow us to compare the Jacobian of the transport map with its
divergence. Since the transport potential need not be twice continuously differentiable, its Hessian must
be understood in the Alexandrov sense, which leads to an
integration-by-parts argument at the level of measures for the full distributional Hessian, including its singular part. Convexity of the
cone controls the boundary contribution produced by the transport map,
while the transversality condition recovers the boundary trace. After
H\"older's inequality and the push-forward relation, a shifted power
profile is chosen as target density. A one-dimensional optimization then
gives the trace estimate and the explicit constant in
Theorem~\ref{thm:trace-continuous}.

The same argument explains why this first estimate need not be sharp. The
weighted Gauss--Green formula contains an additional term involving the
derivative of the weight in the distinguished direction. Under the general
assumptions this term does not vanish, so the boundary step produces a
genuine loss. From a variational point of view, the same obstruction is
reflected in the interior Euler--Lagrange equation: the natural shifted
power profile is not, in general, weighted $p$-harmonic.

In the transport-compatible regime the obstruction disappears. The
condition on the weight removes the additional divergence term, while
the condition on the boundary makes the boundary relation exact. The
transport argument can therefore be repeated with the same boundary estimate.
Choosing the normalized shifted power profile as both source and target
makes the transport map the identity, and the inequalities in the transport
chain become equalities. In particular, the remaining cone boundary contribution vanishes because $x \cdot \eta =0$ on the conical boundary. This gives the optimal constant and the
attainment statement in Theorem~\ref{thm:sharp-trace}.

The rigidity argument follows the same chain in the opposite direction.
If $u$ is an extremal, equality must hold at every stage of the transport
proof. The positivity of a nonnegative extremal allows us to eliminate the
singular part of the Hessian of the convex transport potential. Equality in the
arithmetic--geometric mean step then forces the Hessian to be a constant
multiple of the identity, and hence the transport map is affine,
\[
T(x)=\lambda x+b,
\qquad
\lambda>0.
\]
Since both transported measures have support $\overline\Sigma$, the affine
map preserves the cone. 
 
The Monge--Amp\`ere equation then determines the extremal explicitly, leading
to Theorem~\ref{thm:classification}.

The paper is organized as follows. Section~\ref{sec:preliminaries}
collects the geometric, weighted and analytic facts used throughout the
paper, together with the basic optimal-transport ingredients. In
Section~\ref{sec:proofs} we prove the three main results: first the
continuous trace inequality, then the sharp inequality and attainment in
the transport-compatible regime, and finally the rigidity and complete
classification of extremals.
\section{Preliminaries}
\label{sec:preliminaries}
\subsection{Geometry and homogeneous weights}

We collect here two consequences of the homogeneity and concavity
assumptions that will be used in the proof of the continuous trace
inequality.

The weighted Gauss--Green identity for a sufficiently regular,
compactly supported function $F$ is
\begin{equation}
\label{eq:gauss-green-prelim}
\int_{\partial\Sigma}
Fw\,e\cdot\eta\,\mathrm{d}\sigma
=
\int_\Sigma
e\cdot\nabla F\,w\,\mathrm{d}x
+
\int_\Sigma
F\,e\cdot\nabla w\,\mathrm{d}x.
\end{equation}

Under the geometric assumptions of
Theorem~\ref{thm:trace-continuous}, the last term has a definite sign.
Indeed, the concavity of $w^{1/a}$ and the fact that $-e\in\Sigma$ imply
\begin{equation}
\label{eq:weight-sign-prelim}
e\cdot\nabla w(x)
\leq
-a\,w(-e)^{1/a}w(x)^{1-1/a}
<0
\qquad
\text{for almost every }x\in\Sigma.
\end{equation}

Consequently, if $F$ is nonnegative and not identically
zero in $\Sigma$, then
\[
\int_\Sigma
F\,e\cdot\nabla w\,\mathrm{d}x
<0.
\]

This is the precise source of the loss in the boundary step of the
transport argument for Theorem~\ref{thm:trace-continuous}.

In contrast, for a transport-compatible triple one has
\[
\frac{\partial w}{\partial x_N}=0
\qquad
\text{almost everywhere in }\Sigma,
\]
and therefore \eqref{eq:gauss-green-prelim} reduces to
\[
\int_{\partial\Sigma}
Fw\,e\cdot\eta\,\mathrm{d}\sigma
=
\int_\Sigma
e\cdot\nabla F\,w\,\mathrm{d}x.
\]

Since
\[
e\cdot\eta=c_{\Sigma,e}
\qquad
w\,\mathrm{d}\sigma\text{-almost everywhere on }\partial\Sigma,
\]
we obtain the exact boundary identity
\begin{equation}
\label{eq:exact-boundary-prelim}
c_{\Sigma,e}
\int_{\partial\Sigma}
Fw\,\mathrm{d}\sigma
=
\int_\Sigma
e\cdot\nabla F\,w\,\mathrm{d}x.
\end{equation}

\subsection{Optimal transport and auxiliary results}\label{ssec:opt-transp}

The proofs below use the measure-preserving transport theorem by McCann in~\cite{McCann1995} and the Alexandrov
second-order differentiability theory for convex functions. Let $\mu$ and $\nu$ be two probability measures on $\mathds{R}^{N}$, with $\mu$ absolutely continuous with respect to the Lebesgue measure. The convex-gradient transport map between them
is written as
\[
T=\nabla\varphi,
\]
where $\varphi: \mathds{R}^{N}\to \mathds{R} \cup \{+\infty\}$ is proper, convex and lower semicontinuous such that

\[
T_\# \mu =  \nu
\]
meaning that, whenever $B$ is a Borel subset of $\mathds{R}^{N}$, one has
\[
\nu[B] = \mu[T^{-1}(B)]
\]
or equivalently, for every nonnegative Borel function $b:\mathds{R}^N \to \mathds{R}_+$,
\[
\int b(y)\,\mathrm{d}\nu(y) = \int b(T(x))\,\mathrm{d}\mu(x).
\]
By convexity, $\varphi$ is differentiable almost everywhere on the open convex set
\[
\Omega := \operatorname{int}(\operatorname{dom}\varphi).
\]
Its gradient satisfies
\[
\nabla \varphi \in (\mathrm{BV} \cap L^{\infty})_{\mathrm{loc}} (\Omega;\mathds{R}^{N})
\]
and $\mu$ is concentrated on $\Omega$.

As observed in~\cite{MaggiVillani2005} and also in Lemma~\ref{lem:global-lipschitz-regularity}, if the target measure has compact support, then $\varphi$ can be assumed to be Lipschitz, and, in particular
\[
\Omega = \mathds{R}^N.
\]

When both measures have densities, we use the almost-everywhere Jacobian
formula with the Alexandrov Hessian, as in
\cite{CorderoNazaretVillani2004}. Specifically, if $T=\nabla\varphi$
transports $f\dd x$ to $g\dd y$, then
\begin{equation}\label{eq:jacobian}
 f(x)=g(T(x))\det D_A^2\varphi(x)
 \quad f\dd x\text{-a.e.}
\end{equation}

We shall also use the positivity of nonnegative extremals in the rigidity
argument. Let $u\geq0$ be a nontrivial extremal of
\eqref{eq:sharp-trace}. Since variations compactly supported in $\Sigma$
do not affect the boundary constraint, the interior Euler--Lagrange
equation gives
\begin{equation}
\label{eq:interior-euler-lagrange}
-\operatorname{div}
\left(
w|\nabla u|^{p-2}\nabla u
\right)
=0
\qquad
\text{in }\Sigma
\end{equation}
in the weak sense.

Since $w$ is continuous and strictly positive in the interior of the
cone, for every compact set $K\subset\Sigma$ there exist constants
$0<m_K\leq M_K<\infty$ such that
\[
m_K
\leq
w(x)
\leq
M_K
\qquad
\text{for }x\in K.
\]
Thus
\[
A(x,\xi)
:=
w(x)|\xi|^{p-2}\xi
\]
satisfies locally the standard coercivity and growth conditions
\begin{equation}\label{eq:p-growth-bound}
  A(x,\xi)\cdot\xi
\geq
m_K|\xi|^p,
\qquad
|A(x,\xi)|
\leq
M_K|\xi|^{p-1}.  
\end{equation}
The local regularity theory for quasilinear equations of
$p$-Laplacian type, see for instance~\cite{Serrin1964}, applied to the structural conditions in~\eqref{eq:p-growth-bound}, implies that $u$ is locally Hölder continuous in $\Sigma$.
In particular, for every compact $K\subset\Sigma$, the mollifications $u_\varepsilon$
converge uniformly to $u$ on $K$. Since
$u\ge0$ is a weak nontrivial solution of~\eqref{eq:interior-euler-lagrange}, with the structural bounds~\eqref{eq:p-growth-bound}, the local Harnack inequality
implies $u>0$ in $\Sigma$. 
In particular,
\begin{equation}
\label{eq:local-positive-lower-bound}
\inf_K u>0
\qquad
\text{ for every compact }K\subset\Sigma.
\end{equation}
This local lower bound will be used in the equality analysis to eliminate
the singular part of the Hessian of the convex transport potential.

For $R>0$, we define the truncated domains
\[
\Sigma_R := \Sigma \cap B_R, \qquad B_R : = \{ x \in \mathds{R}^N \colon |x| < R\}.
\]
Throughout the remainder of the paper, we use the notation $\Sigma_R$ for these truncated domains.

\begin{lemma}\label{lem:weight}
For every $R>0$,
\[
 w\in W^{1,1}(\Sigma_R)\cap L^\infty(\Sigma_R),\qquad
 \operatorname{Tr} w=w|_{\partial\Sigma_R}.
\]
If $M_R=\max_{\overline{\Sigma}_R}w$ and $\pi_i$ is the orthogonal projection onto
$e_i^\perp$, then
\begin{equation}\label{eq:bound-dwi} \int_{\Sigma_R}|\partial_iw|\dd x
 \leq 2M_R\HH^{N-1}(\pi_i\Sigma_R), \qquad i=1,2,\cdots,N.   
\end{equation}

\end{lemma}

\begin{proof}
    The function $w$ is locally Lipschitz, even when $0<a<1$, because it is the $a$th power of a finite, concave, and positive function in $\Sigma$. For each fixed $i$, let $\ell_{i}$ be a line parallel to $e_i$. Convexity implies that $I_i=\Sigma_R\cap\ell_i \subset \R$ is an interval whenever it is nonempty.
    On every compact subinterval, the restriction of $w^{1/a}$ is nondecreasing up to a maximum and then nonincreasing with each part possibly empty. For $a>0$ the increasing map $s \mapsto s^a$ preserves this property.
    Hence the variation of $w$ on that subinterval is at most $2M_R$. Exhaustion of $I_i$ by compact subintervals and local absolute continuity give
    \[
    \int_{I_i} |\partial_iw| \dd \HH^1 \le 2 M_R.
    \]
    Fubini's theorem yields~\eqref{eq:bound-dwi}. Since $w$ is bounded on $\Sigma_R$ this proves
    \[
    w \in W^{1,1}(\Sigma_R).
    \]
    Its Sobolev trace agrees with its continuous boundary values.
\end{proof}

\begin{lemma}[Global Lipschitz regularity of the convex transport potential]\label{lem:global-lipschitz-regularity}
Let $\mu\ll\mathcal{L}^N$ be a probability measure and let $\nu$ be a probability measure with
compact support. Let $\varphi_0$ be the convex transport potential transporting
$\mu$ onto $\nu$. 
Then there exists a convex potential $\varphi$, finite and globally Lipschitz on $\R^N$, with
\[
\nabla\varphi=\nabla\varphi_0
\qquad\mu\text{-almost everywhere}.
\]

More precisely, if
\[
\operatorname{supp}\nu\subset \overline{\Sigma} \cap \overline{B}_R,
\]
then
\[
\operatorname{Lip}(\varphi)\leq R.
\]

\end{lemma}

\begin{proof}

Let $K = \overline\Sigma \cap \overline B_R$ and define the indicator function
\[
I_K (x) = \begin{cases}
    0,& \text{if }x \in K\\
    +\infty, & \text{otherwise}, 
\end{cases}
\]
and consider 
\[
\varphi (x)  := \sup_{y \in \R^N}\{x \cdot y - \varphi^*_0(y) - I_K(y)\} = \sup_{y \in K}\{x \cdot y - \varphi^*_0(y)\}
\]
because $I_K \equiv +\infty$ on $\R^N \setminus K$.  
Since $\nabla\varphi_0(x)\in K$ for $\mu$-almost every $x$, we may
choose a differentiability point $x_0$ of $\varphi_0$ and set
\[
z_0:=\nabla\varphi_0(x_0)\in K.
\]
By Fenchel's equality,
\[
\varphi_0^*(z_0)
=
x_0\cdot z_0-\varphi_0(x_0)<+\infty.
\]
In particular,
\[
\varphi(x)
\geq
x\cdot z_0-\varphi_0^*(z_0)>-\infty
\qquad\text{for every }x\in\mathbb{R}^N.
\]
Moreover, by Fenchel's inequality,
\[
\varphi_0^*(z)
\geq
x_0\cdot z-\varphi_0(x_0),
\qquad z\in K.
\]
Therefore, for every $x\in\mathbb{R}^N$ and $z\in K$,
\[
x\cdot z-\varphi_0^*(z)
\leq
(x-x_0)\cdot z+\varphi_0(x_0)
\leq
R|x-x_0|+\varphi_0(x_0),
\]
where we used $K\subset\overline{B_R}$. Taking the supremum over
$z\in K$, we obtain
\[
\varphi(x)
=
\sup_{z\in K}
\left\{x\cdot z-\varphi_0^*(z)\right\}
\leq
R|x-x_0|+\varphi_0(x_0)<+\infty.
\]
Hence
\[
\varphi:\mathbb{R}^N\to\mathbb{R}
\]
is finite.

For any $x,y \in \R^N$ and $z \in K$ we have
\[
x \cdot z -\varphi_0^*(z) = y \cdot z -\varphi_0^*(z) + z \cdot (x-y).
\]
Since $z \cdot (x-y) \le |z||x-y| \le R|x-y|$, taking the supremum over $z \in K$ we obtain
\[
\varphi(x) \le \varphi(y) + R|x-y|.
\]
Interchanging $x$ and $y$,
\[
\varphi(y) \le \varphi(x) + R|x-y|,
\]
concluding that $\mathrm{Lip} (\varphi) \le R$. By Rademacher's theorem $\varphi$ is differentiable almost everywhere on $\R^N$.
Moreover, for $\mu$ almost every $x$
\[
\nabla \varphi_0(x) \in K.
\]
By Fenchel's equality
\[
\varphi_0(x) + \varphi^*_0(\nabla \varphi_0(x))=x \cdot \nabla \varphi_0(x).
\]

On one hand
\begin{equation}\label{eq:fenchel-rel-1}
    \varphi(x) = \sup_{y \in K}\{x \cdot y - \varphi^*_0(y)\} \ge x \cdot \nabla \varphi_0(x) - \varphi^*_0(\nabla \varphi_0(x)) = \varphi_0(x).
\end{equation}
On the other hand
\[
\varphi(x) = \sup_{y \in K}\{x \cdot y - \varphi^*_0(y)\} \le \sup_{y \in \R^N}\{x \cdot y - \varphi^*_0(y)\} = \varphi^{**}_0(x) = \varphi_0(x).
\]
We conclude
\[
\varphi(x)=\varphi_0(x)
\]
for $\dd \mu$ almost every point. Note that for every $z$
\begin{align*}
     \varphi(z) \ge z \cdot \nabla \varphi_0(x) - \varphi_0^*(\nabla \varphi_0(x)) \implies
     \varphi(z) \ge z \cdot \nabla \varphi_0(x) + \varphi_0(x) -x \cdot\nabla \varphi_0(x)  
\end{align*}
Therefore
\[
\varphi(z) \ge \varphi_0(x)+ \nabla \varphi_0(x) \cdot (z-x) =  \varphi(x)+ \nabla \varphi_0(x) \cdot (z-x).
\]
We already know that $\varphi_0(z)-\varphi(z) \ge 0$ for every $z$, and equality holds at the source-almost-everywhere differentiability points considered above. Hence $\varphi_0 - \varphi$ attains its minimum at these differentiability points, and
therefore
\[
\nabla \varphi(x) = \nabla \varphi_0(x), \qquad \mu\text{-almost every } x. \qedhere
\]

\end{proof}

\begin{theorem}\label{thm:greenformula}
Let $\Sigma\subset\mathbb{R}^N$ be an open convex cone, and set
$\Sigma_R=\Sigma\cap B_R$ for $R>0$.
Assume that $w\in C(\overline{\Sigma})$ is positive in $\Sigma$
and homogeneous of degree $a>0$, and that $w^{1/a}$ is concave
in $\Sigma$. Let $\varphi:\mathbb{R}^N\to\mathbb{R}$ be convex and globally
Lipschitz, with
\[
    T=\nabla\varphi\in\overline{\Sigma}
    \qquad\text{Lebesgue-almost everywhere in }\Sigma.
\]
Then, for every $R>0$,
\begin{equation}\label{eq:bv-wT}
    wT\in \BV(\Sigma_R;\mathbb{R}^N).
\end{equation}
Moreover, for every nonnegative
$F\in C_c^1(\mathbb{R}^N)$,
\begin{align}\label{eq:greenformula}
    \int_\Sigma F\,\mathrm{d}[L_w\varphi]
    &=
    -\int_\Sigma wT\cdot\nabla F\,\mathrm{d}x
    +\int_{\partial\Sigma}
        Fw\,\operatorname{Tr}T\cdot\eta\,\mathrm{d}\sigma
    \notag\\
    &\leq
    -\int_\Sigma wT\cdot\nabla F\,\mathrm{d}x,
\end{align}
where $\eta$ is the outward unit normal, defined
$\mathcal{H}^{N-1}$-almost everywhere on $\partial\Sigma$,
and
\begin{equation}\label{eq:decomp-Lw}
    L_w\varphi
    :=
    \operatorname{div}(w\nabla\varphi)
    =
    \bigl(w\Delta_A\varphi+\nabla w\cdot T\bigr)
        \,\mathrm{d}x
    +w\,\operatorname{tr}(D_s^2\varphi).
\end{equation}
Here $D_A^2\varphi$ is the density of the absolutely continuous
part of the distributional Hessian,
$\Delta_A\varphi=\operatorname{tr}(D_A^2\varphi)$,
and $D_s^2\varphi$ is its singular part.
Both measures on the right-hand side of
\eqref{eq:decomp-Lw} are nonnegative.
\end{theorem}

\begin{proof}
The distributional Hessian of a finite convex function
$\varphi$ is a locally finite positive semidefinite
matrix-valued Radon measure. In particular,
\[
    |D^2\varphi|(K)<\infty
\]
for every compact set $K\subset\mathbb{R}^N$, where
$|D^2\varphi|$ denotes its total variation.
Since $\overline{\Sigma}_R$ is compact, this local finiteness
and the global Lipschitz continuity of $\varphi$ yield
\[
    T=\nabla\varphi
    \in \BV(\Sigma_R;\mathbb{R}^N)
       \cap L^\infty(\Sigma_R;\mathbb{R}^N).
\]

The Lebesgue decomposition of the distributional Hessian is
\[
    D^2\varphi
    =
    D_A^2\varphi\,\mathrm{d}x+D_s^2\varphi.
\]
Both parts are positive semidefinite.
Taking the matrix trace gives
\[
    \Delta^D\varphi
    =
    \Delta_A\varphi\,\mathrm{d}x
    +\operatorname{tr}(D_s^2\varphi),
\]
where
\[
    \Delta^D\varphi:=\operatorname{tr}(D^2\varphi),
    \qquad
    \Delta_A\varphi\geq0
    \quad\text{almost everywhere},
    \qquad
    \operatorname{tr}(D_s^2\varphi)\geq0.
\]

Lemma~\ref{lem:weight} gives
\[
    w\in W^{1,1}(\Sigma_R)\cap L^\infty(\Sigma_R).
\]
Moreover, $w$ has a continuous representative.
Since each component $T_i$ belongs to
$\BV(\Sigma_R)\cap L^\infty(\Sigma_R)$,
the product rule for bounded BV functions
(see Ambrosio, Fusco, and Pallara~\cite[Chapter~3]{AFP2000})
implies that $wT_i\in \BV(\Sigma_R)$ and
\begin{equation}\label{eq:productrule}
    D(wT_i)
    =
    w\,DT_i+T_i\nabla w\,\mathrm{d}x,
    \qquad i=1,\ldots,N.
\end{equation}
This identity holds as an equality of vector-valued
Radon measures on $\Sigma_R$.
The factor $w$ in $w\,DT_i$ denotes its continuous
representative; the second term is independent of the
choice of a Lebesgue representative of $T_i$.

Set
\[
    M_R:=\max_{\overline{\Sigma}_R}w.
\]
The product rule gives
\[
    |D(wT)|(\Sigma_R)
    \leq
    M_R|DT|(\Sigma_R)
    +
    \|T\|_{L^\infty(\Sigma_R;\mathbb{R}^N)}
    \int_{\Sigma_R}|\nabla w|\,\mathrm{d}x
    <\infty.
\]
Together with the boundedness of $wT$ on $\Sigma_R$,
this proves \eqref{eq:bv-wT}.
Taking the $i$th component of \eqref{eq:productrule}
and summing over $i$ yields~\eqref{eq:decomp-Lw}.

Choose $R>0$ sufficiently large that the compact set
$\operatorname{supp}F$ is contained in $B_R$.
The bounded convex domain $\Sigma_R$ has Lipschitz boundary.
The BV trace theorem and the continuity of $w$ give
\[
    \operatorname{Tr}(wT)
    =
    w\,\operatorname{Tr}T
    \qquad
    \mathcal{H}^{N-1}\text{-almost everywhere on }
    \partial\Sigma_R.
\]
The Gauss--Green formula for BV vector fields
(see \cite[Chapter~3]{AFP2000}) therefore yields
\[
    \int_{\Sigma_R}F\,\mathrm{d}[L_w\varphi]
    =
    -\int_{\Sigma_R}wT\cdot\nabla F\,\mathrm{d}x
    +
    \int_{\partial\Sigma_R}
        Fw\,\operatorname{Tr}T\cdot\eta_R\,\mathrm{d}\sigma,
\]
where $\eta_R$ is the outward unit normal to $\Sigma_R$.

Since $F$ vanishes in a neighborhood of $\partial B_R$,
the spherical boundary contributes nothing.
On $\partial\Sigma\cap B_R$, the normal $\eta_R$
agrees with $\eta$ almost everywhere. Hence
\[
    \int_\Sigma F\,\mathrm{d}[L_w\varphi]
    =
    -\int_\Sigma wT\cdot\nabla F\,\mathrm{d}x
    +
    \int_{\partial\Sigma}
        Fw\,\operatorname{Tr}T\cdot\eta\,\mathrm{d}\sigma.
\]

We next determine the sign of the boundary term.
Since $T$ takes values in the closed convex set
$\overline{\Sigma}$ almost everywhere, its interior
averages also belong to $\overline{\Sigma}$.
At almost every boundary point, these averages converge
to the BV trace. Thus
\[
    \operatorname{Tr}T(x)\in\overline{\Sigma}
    \qquad
    \mathcal{H}^{N-1}\text{-almost everywhere on }
    \partial\Sigma.
\]
At a boundary point $x$ where the outward normal exists,
convexity gives
\[
    (z-x)\cdot\eta(x)\leq0
    \qquad\text{for every }z\in\overline{\Sigma}.
\]
The cone property implies $x\cdot\eta(x)=0$.
Consequently,
\[
    z\cdot\eta(x)\leq0
    \qquad\text{for every }z\in\overline{\Sigma},
\]
and therefore
\[
    \operatorname{Tr}T\cdot\eta\leq0
    \qquad
    \mathcal{H}^{N-1}\text{-almost everywhere on }
    \partial\Sigma.
\]
Since $F$ and $w$ are nonnegative, this proves
\eqref{eq:greenformula}.

Finally, we verify the nonnegativity of the measures
in \eqref{eq:decomp-Lw}.
Set $h=w^{1/a}$. This function is positive, concave
and homogeneous of degree one in $\Sigma$.
It is therefore locally Lipschitz in $\Sigma$
and differentiable almost everywhere.
At each differentiability point $x\in\Sigma$,
the supporting-plane inequality and Euler's identity give
\begin{equation}\label{concavity}
    h(z)
    \leq
    h(x)+\nabla h(x)\cdot(z-x)
    =
    \nabla h(x)\cdot z,
    \qquad z\in\overline{\Sigma}.
\end{equation}
The inequality first holds for $z\in\Sigma$
and extends to $\overline{\Sigma}$ through continuity.

Since $h\geq0$ on $\overline{\Sigma}$, it follows that
\[
    \nabla h(x)\cdot z\geq0
    \qquad\text{for every }z\in\overline{\Sigma}.
\]
Moreover,
\[
    \nabla w(x)=a\,h(x)^{a-1}\nabla h(x)
\]
almost everywhere in $\Sigma$.
Thus, using $T(x)\in\overline{\Sigma}$, we obtain
\[
    \nabla w(x)\cdot T(x)\geq0
    \qquad\text{almost everywhere in }\Sigma.
\]
Together with $w\geq0$ and $\Delta_A\varphi\geq0$,
this proves that the absolutely continuous part in
\eqref{eq:decomp-Lw} is nonnegative.
The singular part is nonnegative because
$\operatorname{tr}(D_s^2\varphi)$ is a nonnegative measure
and $w$ is continuous and nonnegative.
\end{proof}

\begin{proposition}[Shifted formula]\label{prop:shift}
Under the preceding hypotheses, it holds
\begin{align}\label{eq:shiftformula}
 \int_\Sigma F\dd[L_w\varphi]
 \leq{}&-\int_\Sigma w(T-e)\cdot\nabla F\dd x
 -\int_{\partial\Sigma}Fw e\cdot\eta\dd\sigma
 +\int_\Sigma F e\cdot\nabla w\dd x.
\end{align}
In either of the two geometric regimes
considered above,
\begin{equation}\label{eq:basic}
 \int_\Sigma F\dd[L_w\varphi]
 +c\int_{\partial\Sigma}Fw\dd\sigma
 \leq-\int_\Sigma \nabla F \cdot (T-e)w\dd x,
\end{equation}
with $c=c_\Sigma$ or $c=c_{\Sigma,e}$, respectively.
\end{proposition}

\begin{proof}
The Gauss--Green formula applied to
the vector field $Fw e$ gives
\[
 \int_{\partial\Sigma}Fw e\cdot\eta\dd\sigma
 =\int_\Sigma w e\cdot\nabla F\dd x+
   \int_\Sigma F e\cdot\nabla w\dd x.
\]
Combining this equation with \eqref{eq:greenformula} yields
\[
\int_\Sigma F\dd[L_w\varphi] + \int_{\partial\Sigma}Fw e\cdot\eta\dd\sigma \le -\int_\Sigma wT\cdot\nabla F\dd x + \int_\Sigma we\cdot\nabla F\dd x + \int_\Sigma F e\cdot\nabla w\dd x,
\]
which is exactly the shifted formula~\eqref{eq:shiftformula}.

In the transverse regime, we have
\[
c_\Sigma
:=
\operatorname*{ess\,inf}_{x\in\partial\Sigma\setminus\{0\}}
e\cdot\eta(x)
>0,
\]
and equation~\eqref{concavity} with $z=-e$ yields

\[
 e\cdot\nabla w (x) \leq-a w(-e)^{1/a}w(x)^{1-1/a}<0.
\]
These conditions imply
\[
\int_\Sigma
F\,e\cdot\nabla w\,\mathrm{d}x
<0, \qquad c_\Sigma\int_{\partial\Sigma}Fw \dd\sigma \le \int_{\partial\Sigma}Fw e\cdot\eta\dd\sigma,
\]
for nonnegative $F$ that is not identically
zero in $\Sigma$.

In the compatible regime
\[
\int_\Sigma
F\,e\cdot\nabla w\,\mathrm{d}x
=0,\qquad c_{\Sigma,e}\int_{\partial\Sigma}Fw \dd\sigma = \int_{\partial\Sigma}Fw e\cdot\eta\dd\sigma.
\]
Here 
\[
c_{\Sigma,e} = e \cdot\eta(x), \qquad w\dd \sigma\text{-a.e. on } \partial \Sigma,
\]
and  $e \cdot \nabla w (x) = 0$ a.e. in $\Sigma$.
Inserting these conditions in equation~\eqref{eq:shiftformula} we obtain~\eqref{eq:basic}, for both geometric regimes.
\end{proof}

\begin{proposition}[Algebraic estimate for
the transport map]\label{prop:algebra-transport-map}
Assume the common hypotheses from Section~\ref{sec:main-results} and suppose that one of the two geometric regimes described there holds.
Let $F$ and $G$ be nonnegative measurable functions such that
\[
F^{1-\frac1d}\in C_c^1(\overline{\Sigma})
\qquad\text{and}\qquad
G^{1-\frac1d}\in L^1(\Sigma,w\,\mathrm{d}x),
\]
where $C^{1}_c(\overline \Sigma)$ denotes the restrictions to $\overline \Sigma$ of functions in $C^{1}_c(\R^N)$. In addition, assume that
\[
\mathrm{d}\mu(x)=F(x)w(x)\,\mathrm{d}x,
\qquad
\mathrm{d}\nu(y)=G(y)w(y)\,\mathrm{d}y
\]
are probability measures in $\mathbb{R}^N$, with $\nu$ compactly
supported.
Let $T:=\nabla \varphi_0$ be the map transporting $\mu$ onto $\nu$. Then,

\begin{equation}\label{eq:tran-ineq}
    d \int_{\Sigma} G^{1-1/d}\, w \dd x + c \int_{\partial \Sigma} F^{1-1/d}\, w \dd \sigma \leq -\int_{\Sigma} \nabla (F^{1-1/d}) \cdot (T-e)\, w \dd x
\end{equation}
with $c=c_\Sigma$ or $c=c_{\Sigma,e}$, depending on the geometric regime.
\end{proposition}

\begin{proof}
By McCann's theorem, for every $\nu$-integrable function $\Psi$
\begin{equation}\label{eq:push-fwd-st}
    \int_{\Sigma}
    \Psi(T(x))
    F(x)w(x)\, \mathrm{d} x
    =
    \int_{\Sigma}
    \Psi(y)G(y)w(y)\, \mathrm{d} y.
\end{equation}
Because $\nu$ is supported in $\overline{\Sigma}$ and $T_{\#}\mu=\nu$, we have
\[
    T(x) \in \overline\Sigma, \qquad \mu\, \text{-a.e. } x \in  \Sigma
\]
Since $\operatorname{supp}\nu$ is compact, choose $R>0$ such that
\[
\operatorname{supp}\nu\subset
K:=\overline{\Sigma}\cap\overline{B}_R.
\]
Let $I_K$
denote the indicator function of $K$. We define
\[
\varphi
:=
\bigl(\varphi_0^*+I_K\bigr)^*
=
\sup_{y\in K}
\{x\cdot y-\varphi_0^*(y)\}.
\]
Then $\varphi$ is convex and globally Lipschitz, and
\[
\partial\varphi(x)\subset K.
\]
Consequently,
\[
\nabla\varphi(x)\in\overline{\Sigma}
\qquad\text{for $\mathcal L^{N}$-a.e. }x\in\mathbb R^N.
\]
Moreover, $\varphi=\varphi_0$ $\mu$-a.e., and hence
$\nabla\varphi=\nabla\varphi_0$ $\mu$-a.e.
The corresponding Monge--Amp\`ere equation reads
\begin{equation}\label{eq:monge-amp-st}
    G(T(x))
    w(T(x))
    \det D_A^2\varphi(x)
    =
    F(x)w(x), \qquad \mu\, \text{-a.e. } x \in  \Sigma
\end{equation}
where $D_A^2\varphi$ denotes the
Alexandrov Hessian of $\varphi$.
Define
\[
\Psi_k(y) := \min\left\{G^{-1/d}(y),k\right\}\mathbf{1}_{\{G>0\}}.
\]
Applying \eqref{eq:push-fwd-st} to $\Psi_k$,
\[
\int_{\Sigma} \Psi_k\!\left(T(x)\right) F(x)w(x)\,\mathrm{d} x
=
\int_{\Sigma} \Psi_k(y)G(y)w(y)\,\mathrm{d} y.
\]
By monotone convergence, we obtain
\[
\int_{\Sigma}
    G^{1-\frac1d}(y)w(y)\,\dd y
    =
    \int_{\Sigma}
    G^{-\frac1d}(T(x))
    F(x)w(x)\,\dd x,
\]
where the right-hand side is understood on the set
\[ \lbrace G(T)>0\rbrace.\]
Combining this identity with the Monge--Amp\`ere relation~\eqref{eq:monge-amp-st} gives

\begin{align}
    \int_{\Sigma}
    G^{1-\frac1d}(y)w(y)\, \mathrm{d} y
    &=
    \int_{\Sigma}
    G^{-\frac1d}(T(x))
    F(x)w(x)\, \mathrm{d} x
    \notag\\
    &=
    \int_{\Sigma}
    F^{1-\frac1d}
    \left(
        \frac{w(T)}{w}
    \right)^{1/d}
    \bigl(\det D_A^2\varphi\bigr)^{1/d}
    w\, \mathrm{d} x.
    \label{eq:tran-step-0}
\end{align}

We now exploit the structural assumptions on the weight. Since $w$ is
positively homogeneous of degree $a>0$ and $w^{1/a}$ is concave in
$\Sigma$, it follows that

\begin{equation}\label{eq:weight-conc-st}
    a
    \left(
        \frac{w(z)}{w(x)}
    \right)^{1/a}
    \leq
    \frac{\nabla w(x)\cdot z}{w(x)}
\end{equation}
for almost every $x\in\Sigma$ and for every $z\in\Sigma$. Indeed, by Euler's identity applied to the one-homogeneous function $h=w^{1/a}$
\[
\nabla h(x) \cdot x = h(x) \implies  \nabla w(x) \cdot x = a\,  w(x)
\]
for almost every $x \in \Sigma$. Concavity of $h$ gives
\[
h(z) - h(x) \le \nabla h (x) \cdot (z-x) \implies w^{1/a}(z) \le \frac 1a w^{\frac 1a-1}(x) \nabla w (x) \cdot z
\]
for almost every $x \in \Sigma$ and every $z \in \Sigma$. Equation~\eqref{eq:weight-conc-st} follows from this.

On the other hand, since $D_A^2\varphi$ is nonnegative definite,
the arithmetic--geometric mean inequality gives
\[
    \bigl(\det D_A^2\varphi\bigr)^{1/N}
    \leq
    \frac{\Delta_A\varphi}{N}.
\]
Since $d=N+a$, the arithmetic--geometric mean inequality with weights $\frac ad$ and $\frac Nd$
followed by \eqref{eq:weight-conc-st}, yields
\begin{align*}
&
\left(
    \frac{w(T)}{w}
\right)^{1/d}
\bigl(\det D_A^2\varphi\bigr)^{1/d}
=
\left[
    \left(
        \frac{w(T)}{w}
    \right)^{1/a}
\right]^{a/d}
\left[
    \bigl(\det D_A^2\varphi\bigr)^{1/N}
\right]^{N/d}
\\
&\quad\leq
\frac{a}{d}
\left(
    \frac{w(T)}{w}
\right)^{1/a}
+
\frac{N}{d}
\bigl(\det D_A^2\varphi\bigr)^{1/N}
\\
&\quad\leq
\frac{1}{d}
\left(
    \frac{\nabla w}{w}\cdot T
    +
    \Delta_A\varphi
\right).
\end{align*}

Consequently,

\begin{equation}\label{eq:transport-step}
    d\int_{\Sigma} G^{1-\frac{1}{d}}\,w \dd x
    \leq
    \int_{\Sigma} F^{1-\frac{1}{d}}
    \left(
    \Delta_A\varphi+\frac{\nabla w}{w}\cdot\nabla\varphi
    \right)w\,\mathrm{d}x.
\end{equation}

In view of Theorem~\ref{thm:greenformula}, Proposition~\ref{prop:shift} and the fact that the singular part of the distributional Hessian of $\varphi$ is nonnegative ($D^2_s \varphi \ge 0$) as a matrix-valued measure, the right-hand side of equation~\eqref{eq:transport-step} is bounded above by
\[
\int_{\Sigma} F^{1-\frac{1}{d}}
    \dd [L_w\varphi]
\]
and therefore
\begin{equation}\label{eq:transport-setp-2}
    d\int_{\Sigma} G^{1-\frac{1}{d}}\,w \dd x \le \int_{\Sigma} F^{1-\frac{1}{d}}
    \dd [L_w\varphi].
\end{equation}
The shifted formula in~\eqref{eq:basic} yields
\[
\int_{\Sigma} F^{1-\frac{1}{d}}
    \dd [L_w\varphi] + c \int_{\partial \Sigma} F^{1-1/d}\, w \dd \sigma \leq -\int_{\Sigma} \nabla (F^{1-1/d}) \cdot (T-e)\, w \dd x.
\]
These last two assertions together with~\eqref{eq:transport-setp-2} conclude the proof.
\end{proof}

\begin{proposition}[Interior realization of the energy space]\label{prop:energy-space}
There exists a constant $C_I>0$ such that, for every $f\in C_c^\infty(\overline \Sigma)$,
\begin{equation}\label{eq:interior}
    \|f\|_{L^{p^*}(\Sigma,w)}
\le C_I\|\nabla f\|_{L^p(\Sigma,w)}.
\end{equation}
Consequently, the completion of $C_c^\infty(\overline \Sigma)$ with respect to the energy seminorm
\[
\|f\|_{\dot W^{1,p}}
:=
\|\nabla f\|_{L^p(\Sigma,w)}
\]
is continuously and injectively realized as a subspace of
\[
L^{p^*}(\Sigma,w)\cap W_{\mathrm{loc}}^{1,p}(\Sigma,w).
\]
More precisely, if $(f_j)$ is Cauchy with respect to $\|\cdot\|_{\dot W^{1,p}}$, then there exists a unique
\[
u\in L^{p^*}(\Sigma,w)\cap W_{\mathrm{loc}}^{1,p}(\Sigma,w)
\]
such that
\[
f_j\to u
\quad\text{in }L^{p^*}(\Sigma,w),
\qquad
\nabla f_j\to\nabla u
\quad\text{in }L^p(\Sigma,w).
\]
\end{proposition}

\begin{proof}

Fix a nonnegative $v_0\in C_c^\infty(\Sigma)$ with
$\int v_0^{p^*}w=1$. Define
\[
 V_0=\int_\Sigma v_0^{p_*}w\dd x>0,\qquad
 M_0=\int_\Sigma |y|^{p'}v_0^{p^*}w\dd y\in(0,\infty).
\]
Let $u \in C_c^{\infty}(\overline \Sigma)$ be a nonnegative function, such that
\[
\int_{\Sigma} u^{p^*}w\,\dd x = 1
\]
Set
\[
    F:=u^{p^*},
    \qquad
    G:=v_0^{p^*}.
\]
Note that $F^{1-1/d} = u^{p_*} \in C_c^{1}(\overline \Sigma)$, because $p_*>1$. Using equation~\eqref{eq:transport-setp-2} from the proof of Proposition~\ref{prop:algebra-transport-map} together with formula~\eqref{eq:greenformula} from Theorem~\ref{thm:greenformula},

\[
d\int_{\Sigma} G^{1-\frac{1}{d}}\,w \dd x \le \int_{\Sigma} F^{1-\frac{1}{d}}
    \dd [L_w\varphi] \le -\int_\Sigma wT\cdot\nabla (F^{1-\frac1d})\dd x.
\]
Note that
\[
G^{1-\frac1d} = v_0^{p_*}, \qquad F^{1-\frac1d} = u^{p_*}
\]
we obtain 
\[
 dV_0\le -p_*\int_\Sigma u^{p_*-1}\nabla u\cdot T\,w\dd x
 \le p_*M_0^{1/p'}\|\nabla u\|_{L^p(w)}
\]
using H\"older's inequality and the following push-forward identity
\[
    \int_{\Sigma}
    u^{p^*}
    |\nabla\varphi|^{p'}
    w\, \mathrm{d} x
    =
    \int_{\Sigma}
    v_0^{p^*}
    |x|^{p'}
    w\, \mathrm{d} x.
\]
Amplitude normalization gives
\begin{equation}\label{eq:interior-embed}
 \|f\|_{L^{p^*}(w)}\leq C_I\|\nabla f\|_{L^p(w)},\qquad
 C_I=\frac{p_*M_0^{1/p'}}{dV_0}.
\end{equation}
For signed $f$, apply the argument to
\[
f_{\varepsilon} = \sqrt{f^2+\varepsilon^2}-\varepsilon
\]
and let $\varepsilon\downarrow0$.
This function is smooth and compactly supported, its gradient is bounded
by $|\nabla f|$, and it converges to $|f|$. No boundary trace estimate
has been used. 
If $(f_j)$ is energy-Cauchy, \eqref{eq:interior-embed} gives a limit
$u\in L^{p^*}(\Sigma,w)$ and a gradient limit in $L^p(\Sigma,w)$. On every compact
interior set, $w$ is bounded above and below by positive constants.
Local unweighted convergence and distributional integration against
compactly supported tests identify that limit as $\nabla u$.
Consequently the completion is realized injectively as functions in
$L^{p^*}(\Sigma,w)\cap W^{1,p}_{\mathrm{loc}}(\Sigma,w)$. In particular,
\begin{equation}\label{energyapprox}
 f_j\to u\text{ in }L^{p^*}(\Sigma,w),\qquad
 \nabla f_j\to\nabla u\text{ in }L^p(\Sigma,w).
\end{equation}
The gradient seminorm is a norm on this realization by
\eqref{eq:interior-embed}. 
\end{proof}

The following auxiliary results will be used only in the proof of Theorem~\ref{thm:classification}, once we have established Theorems~\ref{thm:trace-continuous} and~\ref{thm:sharp-trace}.

\begin{proposition}[General sources and bounded-gradient potentials]
\label{prop:bounded-general}
Assume equation~\eqref{eq:interior}, the trace estimate~\eqref{eq:sharp-trace} in the compatible regime, and
transport compatibility. Let
\[
 u\in\dot W^{1,p}(\Sigma,w),
 \qquad u\geq0.
\]
Let $\varphi$ satisfy the hypotheses of Theorem~\ref{thm:greenformula}. Then
\begin{align}\label{bounded-general}
 &\int_\Sigma
 u^{p_*}
 \bigl(w\Delta_A\varphi+\nabla w\cdot\nabla\varphi\bigr)\dd x
 \notag\\
 &\quad+
 c_{\Sigma,e}
 \int_{\partial\Sigma}(\Tr u)^{p_*}w\dd\sigma
 \notag\\
 &\qquad\leq
 -p_*\int_\Sigma
 u^{p_*-1}\nabla u\cdot(\nabla\varphi-e)w\dd x.
\end{align}
\end{proposition}

\begin{proof}
Assuming the interior and trace estimates, there exists a sequence of functions $(u_j)_{j \in \mathbb{N}}$ in $C^{\infty}_c(\overline \Sigma)$, $u_j \ge 0$, such that
\begin{equation}\label{eq:approximation}
\begin{gathered}
 u_j\to u
 \quad\text{in }L^{p^*}(\Sigma,w),\\
 \nabla u_j\to\nabla u
 \quad\text{in }L^p(\Sigma,w),\\
 \Tr u_j\to\Tr u
 \quad\text{in }L^{p_*}(\partial\Sigma,w\dd\sigma).
\end{gathered}
\end{equation}
Because $u\geq0$, these approximations can be chosen nonnegative and $C^1$, by using the positive-part operation and smooth approximations to it. 
Choose a radial, nonincreasing function
$\chi\in C_c^\infty(\R^N)$ such that
\[
0\leq \chi\leq 1,\qquad
\chi\equiv 1 \text{ on } B_1,\qquad
\chi\equiv 0 \text{ on } \R^N\setminus B_2.
\]
For $R>0$, define
\[
\chi_R(x):=\chi\left(\frac{x}{R}\right).
\]
Then
\[
\chi_R=1 \text{ on } B_R,\qquad
\chi_R=0 \text{ on } \R^N\setminus B_{2R},
\qquad
|\nabla\chi_R|\leq \frac{C}{R}.
\]
Moreover, since $\chi$ is radially nonincreasing,
\[
\chi_R(x)\leq \chi_{R'}(x)
\qquad\text{whenever }R\leq R',
\]
and hence
\[
\chi_R(x)\uparrow 1
\qquad\text{as }R\to\infty
\]
for every $x\in\R^N$.
Applying formula~\eqref{eq:basic} 
\[
\int_\Sigma F\dd[L_w\varphi]
 +c_{\Sigma,e}\int_{\partial\Sigma}Fw\dd\sigma
 \leq-\int_\Sigma \nabla F \cdot (\nabla \varphi-e)w\dd x
\]
for $F = \chi_R u_j^{p_*} \in C^{1}_c(\R^N)$ and some fixed $R>0$ implies
\begin{align*}\label{}
 &\int_\Sigma
 \chi_Ru_j^{p_*}
 \dd[L_w\varphi]
+
 c_{\Sigma,e}
 \int_{\partial\Sigma}\chi_R(\Tr u_j)^{p_*}w\dd\sigma\\
 &\leq
 -p_*\int_\Sigma
\chi_R u_j^{p_*-1}\nabla u_j\cdot(\nabla\varphi-e)w\dd x - \int_{E_R}
 u_j^{p_*}\nabla \chi_R\cdot(\nabla\varphi-e)w\dd x,
\end{align*}
where $E_R:=\Sigma \cap(B_{2R} \setminus B_R)$. Since the singular part of the distributional Hessian of $\varphi$ is nonnegative ($D^2_s \varphi \ge 0$) as a matrix-valued measure, we obtain
\begin{align*}\label{}
 &\int_\Sigma
 \chi_Ru_j^{p_*}
 (w \Delta_A \varphi+\nabla w \cdot \nabla \varphi) \dd x
+
 c_{\Sigma,e}
 \int_{\partial\Sigma}\chi_R(\Tr u_j)^{p_*}w\dd\sigma\\
 &\leq
 -p_*\int_\Sigma
\chi_R u_j^{p_*-1}\nabla u_j\cdot(\nabla\varphi-e)w\dd x - \int_{E_R}
 u_j^{p_*}\nabla \chi_R\cdot(\nabla\varphi-e)w\dd x.
\end{align*}

Now let $j \to \infty$. On the left-hand side, Fatou's lemma resolves the first integral while the second integral follows from $|\chi_R| \le 1$ and the convergence in $L^{p_*}(\partial\Sigma,w\dd \sigma)$. On the right-hand side, the first integral converges because $(p_*-1)p'=p^*$,~\eqref{eq:approximation} implies $u_j^{p_*-1} \to u^{p_*-1} \text{ in } L^{p'}(\Sigma,w)$, $\nabla u_j \to \nabla u \text{ in } L^{p}(\Sigma,w)$, and $\nabla \varphi \in L^{\infty}$. The remaining integral converges by local $L^{1}(\Sigma,w)$ convergence of $u^{p_*}_j$. Therefore,
\begin{align*}\label{}
 &\int_\Sigma
 \chi_Ru^{p_*}
 (w \Delta_A \varphi+\nabla w \cdot \nabla \varphi) \dd x
+
 c_{\Sigma,e}
 \int_{\partial\Sigma}\chi_R(\Tr u)^{p_*}w\dd\sigma\\
 &\leq
 -p_*\int_\Sigma
\chi_R u^{p_*-1}\nabla u\cdot(\nabla\varphi-e)w\dd x - \int_{E_R}
 u^{p_*}\nabla \chi_R\cdot(\nabla\varphi-e)w\dd x.
\end{align*}

Finally, let $R \to \infty$. On the left-hand side, the monotone convergence theorem resolves the first integral, while the second integral follows from the dominated convergence theorem since we have
\[
\chi_R (\Tr u)^{p_*} w \to (\Tr u)^{p_*}w \text{ a.e., and} \qquad  (\Tr u)^{p_*} \in L^1(\partial \Sigma,w \dd \sigma).
\]
On the right-hand side, the first integral follows from H\"older's inequality and the dominated convergence theorem. Homogeneity gives
\[
\int_{E_R} w \dd x \le C R^{d},
\]
hence by boundedness of $\nabla \varphi$,
\[
\left |\int_{E_R}
 u^{p_*}\nabla \chi_R\cdot(\nabla\varphi-e)w\dd x \right |\le \frac CR \int_{E_R} u^{p_*} w \dd x \le \frac CR \left (\int_{E_R} u^{p^*} w \dd x \right )^{p_*/p^{*}} \left ( \int_{E_R} w \dd x \right )^{1-p_*/p^{*}}.
\]
The last term is bounded above by
\[
C \left (\int_{E_R} u^{p^*} w\dd x \right )^{p_*/p^{*}} \leq C \left (\int_{\Sigma \setminus B_R} u^{p^*} w\dd x \right )^{p_*/p^{*}}\to 0. \qedhere
\]

\end{proof}

\begin{lemma}[Lipschitz truncation of the transport potential]\label{lem:trucation-potential}
Let $\mu$ and $\nu$ be two probability measures on $\mathds{R}^{N}$, with support on $\overline \Sigma$, $\mu$ absolutely continuous with respect to Lebesgue measure and $\nu$ not necessarily compactly supported.
Let $\varphi_0:\mathds{R}^{N}\to \mathds{R} \cup \{+\infty\}$ be the convex transport potential from $\mu$ onto $\nu$. Define
\[
\varphi=(\varphi_0^*+I_{\overline \Sigma})^*,
\]
where $I_A$ denotes the indicator function of a set $A$.
For $k\in\mathbb{N}$, define
\[
\varphi_k=(\varphi^*+I_{\overline \Sigma\cap \overline B_k})^*.
\]
Then there exists $k_0 \in \mathds{N}$ such that for every $k \ge k_0$, $\varphi_k$ is finite and globally $k$-Lipschitz. Define $T_k:=\nabla\varphi_k$. If
\[
T=\nabla\varphi,
\]
then
\[
T_k\to T
\quad\text{a.e. in }\Omega,\qquad \Omega=\operatorname{int}(\operatorname{dom}\varphi).
\]
Moreover, for every compact set $K\subset\Omega$, there exists
$k_1=k_1(K) \ge k_0$ such that, for all $k\geq k_1$,
\[
\varphi_k=\varphi
\quad\text{in a neighborhood of }K.
\]
\end{lemma}
\begin{proof}
As in the proof of Lemma~\ref{lem:global-lipschitz-regularity}, the definition
\[
\varphi=(\varphi_0^*+I_{\overline \Sigma})^*,
\]
leaves the transport gradient
unchanged at $\mu$-almost every point. Moreover,
\[
 \operatorname{dom}\varphi^*
 \subset\overline\Sigma, \qquad \operatorname{dom}\varphi^*\neq \emptyset.
\]
Then, there exists a $k_0 \in \mathbb{N}$ such that
\[
(\overline\Sigma \cap \overline B_{k_0}) \cap \operatorname{dom}\varphi^* \neq \emptyset.
\]
Note that $\Omega$ is nonempty and carries the source measure: a convex
domain carrying an absolutely continuous probability measure is
full-dimensional, and its boundary has Lebesgue measure zero.

For $k \ge k_0$, define
\begin{equation}\label{dual-truncation}
 \varphi_k(x)
 :=
 \sup_{y\in\overline\Sigma\cap\overline B_k}
 \{x\cdot y-\varphi^*(y)\}.
\end{equation}
These potentials are finite, convex and globally $k$-Lipschitz.
Their gradients satisfy
\[
 \nabla\varphi_k(x)\in\overline\Sigma
 \quad\text{for almost every }x\in\R^N.
\]
For every compact $K\subset \Omega$, local boundedness of the subgradients of
$\varphi$ gives
\begin{equation}\label{local-equality}
 \varphi_k=\varphi
 \quad\text{on a neighborhood of }K
 \quad\text{for sufficiently large }k.
\end{equation}
Indeed, each supporting slope there is bounded, belongs to
$\operatorname{dom}\varphi^*$, and is admitted in~\eqref{dual-truncation}. Fenchel's equality proves the assertion as in the proof of Lemma~\ref{lem:global-lipschitz-regularity}.
At a differentiability point of $\varphi$, if
$|\nabla\varphi(x)|\leq k$, then $\nabla\varphi_k(x)=\nabla\varphi(x)$. Otherwise, $|\nabla\varphi_k(x)|\leq k<|\nabla\varphi(x)|.$
In any case
\begin{equation}\label{domination}
\begin{gathered}
 |\nabla\varphi_k(x)|
 \leq|\nabla\varphi(x)|,\\
 \nabla\varphi_k(x)\to\nabla\varphi(x)
 \quad\text{for almost every }x\in \Omega. 
\end{gathered}
\qedhere
\end{equation}

\end{proof}

\begin{theorem}[]
\label{thm:general}
Assume equation~\eqref{eq:interior}, the trace estimate~\eqref{eq:sharp-trace} in the compatible regime, and
transport compatibility. Let
\[
 u,v\in\dot W^{1,p}(\Sigma,w),
 \qquad u,v\geq0,
\]
satisfy
\[
 \|u\|_{L^{p^*}(\Sigma,w)}
 =
 \|v\|_{L^{p^*}(\Sigma,w)}
 =
 1
\]
and
\[
 \int_\Sigma |y-e|^{p'}v^{p^*}(y)w(y)\dd y<\infty.
\]
For the potential chosen as in Lemma~\ref{lem:trucation-potential},
\begin{align}\label{general}
 &\int_{\Omega\cap\Sigma}
 u^{p_*}
 \bigl(w\Delta_A\varphi+\nabla w\cdot\nabla\varphi\bigr)\dd x
 \notag\\
 &\quad+
 c_{\Sigma,e}
 \int_{\partial\Sigma}(\Tr u)^{p_*}w\dd\sigma
 \notag\\
 &\qquad\leq
 -p_*\int_\Sigma
 u^{p_*-1}\nabla u\cdot(\nabla\varphi-e)w\dd x.
\end{align}
Products involving $\nabla\varphi$ outside $\Omega$ are defined to be
zero where $u=0$. Indeed
\[
u = 0 \text{ a.e. on } \Sigma \setminus\Omega,
\]
because the source measure $u^{p^*}w\dd x$ is concentrated on $\Omega$ and $w>0$ in $\Sigma$.
Consequently,
\begin{align}\label{lemma22}
 d\int_\Sigma v^{p_*}w\dd x
 \leq{}&
 -p_*\int_\Sigma
 u^{p_*-1}\nabla u\cdot(\nabla\varphi-e)w\dd x
 \notag\\
 &-
 c_{\Sigma,e}
 \int_{\partial\Sigma}(\Tr u)^{p_*}w\dd\sigma.
\end{align}
\end{theorem}

\begin{proof}
Applying Proposition~\ref{prop:bounded-general} to $\varphi_k$, we obtain
\begin{align*}
 &\int_\Sigma
 u^{p_*}
 \bigl(w\Delta_A\varphi_k+\nabla w\cdot\nabla\varphi_k\bigr)\dd x
+
 c_{\Sigma,e}
 \int_{\partial\Sigma}(\Tr u)^{p_*}w\dd\sigma
 \notag\\
 &\leq
 -p_*\int_\Sigma
 u^{p_*-1}\nabla u\cdot(\nabla\varphi_k-e)w\dd x.
\end{align*}
The densities on the left are nonnegative. By \eqref{local-equality}, their absolutely continuous parts
eventually agree with the corresponding density for $\varphi$
at almost every point of $\Omega\cap\Sigma$. Fatou's lemma therefore
gives the left side of \eqref{general} in the limit.

The push-forward identity gives
\[
 \int_\Sigma
 u^{p^*}|\nabla\varphi-e|^{p'}w\dd x
 =
 \int_\Sigma
 v^{p^*}|y-e|^{p'}w\dd y
 <
 \infty.
\]
Set
\[
 \delta:=\operatorname{dist}(e,\overline\Sigma)>0.
\]
In the source set, \eqref{domination} implies
\[
 \begin{split}
 |\nabla\varphi_k-e|
 &\leq|\nabla\varphi|+|e|\\
 &\leq|\nabla\varphi-e|+2|e|\\
 &\leq
 \left(1+\frac{2|e|}{\delta}\right)
 |\nabla\varphi-e|.
 \end{split}
\]
Since $(p_*-1)p'=p^*$, dominated convergence yields
\[
 u^{p_*-1}(\nabla\varphi_k-e)
 \longrightarrow
 u^{p_*-1}(\nabla\varphi-e)
 \quad\text{in }L^{p'}(\Sigma,w).
\]
Pairing with $\nabla u\in L^p(\Sigma,w)$ proves the convergence of
the right-hand side.

The Monge--Amp\`ere and weighted arithmetic--geometric mean
calculation gives
\[
 d\int_\Sigma v^{p_*}w\dd x
 \leq
 \int_{\Omega\cap\Sigma}
 u^{p_*}
 \bigl(w\Delta_A\varphi+\nabla w\cdot\nabla\varphi\bigr)\dd x.
\]
Combining this with~\eqref{general} proves~\eqref{lemma22}.

For completeness, the integral of $v^{p_*}$ is finite.
Write
\[
 v^{p_*}
 =
 \bigl(v^{p^*}|y-e|^{p'}\bigr)^{p_*/p^*}
 |y-e|^{-p'p_*/p^*}
\]
and apply H\"older's inequality. Since
\[
 \frac{p_*}{p^*-p_*}=d-1,
\]
the remaining factor is controlled by
\[
 \int_\Sigma
 |y-e|^{-(d-1)p'}w(y)\dd y<\infty.
\]
The positive distance from $e$ to the cone excludes a finite-point
singularity. Homogeneity gives convergence at infinity because
\[
 (d-1)p'>d
 \quad\Longleftrightarrow\quad
 p<d. \qedhere
\]
\end{proof}

\section{Proofs of the main results}
\label{sec:proofs}


\subsection{Proof of Theorem~\ref{thm:trace-continuous}}
We first prove the estimate for
$u\in C^\infty_c(\overline\Sigma)$. Since
\[
    |\nabla |u||=|\nabla u|
    \qquad\text{a.e. in }\Sigma,
\]
and the trace norm depends only on $|u|$, a standard approximation argument
allows us to restrict attention to nonnegative functions.

Recall that
\[
    d=N+a,
    \qquad
    p^*=\frac{pd}{d-p},
    \qquad
    p_*=\frac{p(d-1)}{d-p},
    \qquad
    p'=\frac{p}{p-1}.
\]

We shall repeatedly use the identities
\begin{equation}\label{eq:exp-identities}
    p^*\left(1-\frac1d\right)=p_*,
    \qquad
    (p_*-1)p'=p^*.
\end{equation}
Let $u,v\in C^\infty_c(\overline\Sigma)$ be nonnegative functions with
\begin{equation}\label{eq:normalization-st}
    \|u\|_{L^{p^*}(\Sigma,w)}
    =\|v\|_{L^{p^*}(\Sigma,w)}=1,\qquad \|u\|_{L^{p^*}(\Sigma,w)} := \left ( \int_{\Sigma}
    u^{p^*}(x)w(x)\, \mathrm{d} x \right)^{1/p^*}.
\end{equation}

Set
\[
    F:=u^{p^*},
    \qquad
    G:=v^{p^*}.
\]
Note that $F^{1-1/d} = u^{p_*} \in C_c^{1}(\overline\Sigma)$, because $p_*>1$. Applying Proposition~\ref{prop:algebra-transport-map} and using identities~\eqref{eq:exp-identities},
\begin{align}
&
c_\Sigma
\int_{\partial\Sigma}
u^{p_*}w\, \mathrm{d}\sigma
+
d
\int_{\Sigma}
v^{p_*}w\, \mathrm{d} x
\notag\\
&\qquad\leq
-p_*
\int_{\Sigma}
u^{p_*-1}
\nabla u\cdot(\nabla\varphi-e)
w\, \mathrm{d} x.
\label{THIRD-MASTER-ST}
\end{align}
H\"older's inequality and the identity
$(p_*-1)p'=p^*$ now give
\begin{align}
&
c_\Sigma
\|u\|_{L^{p_*}(\partial\Sigma,w)}^{p_*}
+
d
\|v\|_{L^{p_*}(\Sigma,w)}^{p_*}
\notag\\
&\qquad\leq
p_*
\|\nabla u\|_{L^p(\Sigma,w)}
\left(
    \int_{\Sigma}
    u^{p^*}
    |\nabla\varphi-e|^{p'}
    w\, \mathrm{d} x
\right)^{1/p'}.
\label{eq:holder-st}
\end{align}

By the push-forward identity,
\[
    \int_{\Sigma}
    u^{p^*}
    |\nabla\varphi-e|^{p'}
    w\, \mathrm{d} x
    =
    \int_{\Sigma}
    v^{p^*}
    |x-e|^{p'}
    w\, \mathrm{d} x.
\]
We therefore obtain the fundamental estimate
\begin{align}
&
c_\Sigma
\|u\|_{L^{p_*}(\partial\Sigma,w)}^{p_*}
+
d
\|v\|_{L^{p_*}(\Sigma,w)}^{p_*}
\notag\\
&\qquad\leq
p_*
\|\nabla u\|_{L^p(\Sigma,w)}
\left(
    \int_{\Sigma}
    v^{p^*}|x-e|^{p'}w\, \mathrm{d} x
\right)^{1/p'}.
\label{eq:fundamental-st}
\end{align}
We now choose a suitable target density. Set
\[
    \kappa:=\frac{d-p}{p-1}, \qquad G_0(x):=|x-e|^{-\kappa}.
\]
We first verify the required integrability. Let
\[
    Z
    :=
    \int_{\Sigma}
    G_0^{p^*}w\, \mathrm{d} x
\]
and
\begin{equation}\label{eq:def-I-st}
    I
    :=
    \int_{\Sigma}
    |x-e|^{-(d-1)p'}w(x)\, \mathrm{d} x.
\end{equation}
Since $\Sigma\subset\{x_N>0\}$ while
$e=(0,\ldots,0,-1)$, the point $e$ lies outside
$\overline\Sigma$; hence neither integral has a singularity at a finite
point of $\overline\Sigma$. Moreover, homogeneity and continuity of $w$
on the spherical section imply 
\[
    w(x)\leq C|x|^a,
\]
for every $x \in \Sigma$. For $|x|\geq 2|e|$,
\[
    |x-e|\geq\frac{|x|}{2},
\]
and therefore, 
\begin{align*}
    I
    &\leq
    C
    \int_2^\infty
    r^{N-1+a-(d-1)p'}\, \mathrm{d} r
    \\
    &=
    C
    \int_2^\infty
    r^{d-1-(d-1)p'}\, \mathrm{d} r
    \\
    &=
    C
    \int_2^\infty
    r^{-\frac{d-1}{p-1}}\, \mathrm{d} r.
\end{align*}
Since $p<d$, the last integral converges. Thus
\[
    I<\infty.
\]
Similarly, since
\[
    G_0^{p^*}=|x-e|^{-dp'},
\]
the same argument shows that
\[
    Z<\infty.
\]
Define
\begin{equation}\label{eq:def-G0}
    c_0:=Z^{-1/p^*}, \qquad v:=c_0G_0.
\end{equation}
Then
\[
    \|v\|_{L^{p^*}(\Sigma,w)}=1.
\]
Although $v$ is not compactly supported, it may be inserted into~\eqref{eq:fundamental-st} by approximation. Indeed, let
\[
    v_R:=c_R\chi_R G_0,
\]
where $\chi_R\in C^\infty_c(\mathds{R}^N)$ satisfies
\[
    0\leq\chi_R\leq1,
    \qquad
    \chi_R\equiv1\ \text{on }B_R,
    \qquad
    \chi_R\equiv0\ \text{outside }B_{2R},
\]
and $c_R>0$ is chosen so that
\[
    \|v_R\|_{L^{p^*}(\Sigma,w)}=1.
\]
Then $c_R\to c_0$, and dominated convergence allows us to pass to the
limit in~\eqref{eq:fundamental-st} as $R\to\infty$.

A direct computation gives
\[
    G_0^{p_*}
    =
    |x-e|^{-(d-1)p'},
\]
and
\[
    G_0^{p^*}|x-e|^{p'}
    =
    |x-e|^{-(d-1)p'}.
\]
Consequently,
\[
    \|v\|_{L^{p_*}(\Sigma,w)}^{p_*}
    =
    c_0^{p_*}I,
\]
and
\[
    \left(
        \int_{\Sigma}
        v^{p^*}|x-e|^{p'}w\, \mathrm{d} x
    \right)^{1/p'}
    =
    c_0^{p^*/p'}I^{1/p'}.
\]
Substituting these expressions into~\eqref{eq:fundamental-st}, we obtain,
for every $u$ satisfying~\eqref{eq:normalization-st},
\begin{equation}\label{eq:normalized-trace-st}
    \|u\|_{L^{p_*}(\partial\Sigma,w)}^{p_*}
    \leq
    A\|\nabla u\|_{L^p(\Sigma,w)}-B,
\end{equation}
where
\begin{equation}\label{eq:ab-st}
    A
    :=
    \frac{p_*}{c_\Sigma}
    c_0^{p^*/p'}I^{1/p'},
    \qquad
    B
    :=
    \frac{d}{c_\Sigma}
    c_0^{p_*}I.
\end{equation}

We next remove the normalization. If
$\|u\|_{L^{p_*}(\partial\Sigma,w)}=0$, there is nothing to prove.
Otherwise, define
\[
    \mathcal{J}(u)
    :=
    \frac{\|\nabla u\|_{L^p(\Sigma,w)}}
         {\|u\|_{L^{p_*}(\partial\Sigma,w)}}, \qquad \mathcal{K}(u)
    :=
    \frac{\|\nabla u\|_{L^p(\Sigma,w)}}
         {\|u\|_{L^{p^*}(\Sigma,w)}}.
\]
Applying~\eqref{eq:normalized-trace-st} to $u/\|u\|_{L^{p^*}(\Sigma,w)}$ gives
\[
    \left(
        \frac{\mathcal{K}(u)}{\mathcal{J}(u)}
    \right)^{p_*}
    \leq
    A\mathcal{K}(u)-B,
\]
and hence
\begin{equation}\label{J-LOWER-BOUND}
    \mathcal{J}(u)^{p_*}
    \geq
    \frac{\mathcal{K}(u)^{p_*}}{A\mathcal{K}(u)-B}.
\end{equation}
Consider the function
\[
    \Phi(t):=\frac{t^{p_*}}{At-B},
    \qquad
    t>\frac{B}{A}.
\]
Its derivative is
\[
    \Phi'(t)
    =
    \frac{
        t^{p_*-1}
        \bigl((p_*-1)At-p_*B\bigr)
    }
    {(At-B)^2},
\]
and therefore $\Phi$ attains its minimum at
\[
    t_0
    =
    \frac{p_*B}{(p_*-1)A}.
\]
Since
\[
    p_*-1
    =
    \frac{d(p-1)}{d-p},
\]
we obtain from~\eqref{eq:ab-st}
\[
    t_0
    =
    \frac{d-p}{p-1}
    c_0 I^{1/p}.
\]
Moreover,
\[
    At_0-B=\frac{B}{p_*-1}.
\]
It follows that
\begin{align*}
    \min_{t>B/A}\Phi(t)
    &=
    \frac{t_0^{p_*}}{At_0-B}
    =
    c_\Sigma
    \left(
        \frac{d-p}{p-1}
    \right)^{p_*-1}
    I^{\frac{p-1}{d-p}}.
\end{align*}
Therefore~\eqref{J-LOWER-BOUND} yields
\[
    \mathcal{J}(u)^{p_*}
    \geq
    c_\Sigma
    \left(
        \frac{d-p}{p-1}
    \right)^{p_*-1}
    I^{\frac{p-1}{d-p}}.
\]
Equivalently,
\begin{equation}\label{FINAL-TRACE-ST}
    \left(
        \int_{\partial\Sigma}
        |u|^{p_*}w\, \mathrm{d}\sigma
    \right)^{1/p_*}
    \leq
    C_{\mathrm{MT}}
    \left(
        \int_{\Sigma}
        |\nabla u|^p w\, \mathrm{d} x
    \right)^{1/p},
\end{equation}
where
\begin{equation}\label{eq:cmt-correct-st}
    C_{\mathrm{MT}}
    =
    \left[
        c_\Sigma
        \left(
            \frac{d-p}{p-1}
        \right)^{p_*-1}
        I^{\frac{p-1}{d-p}}
    \right]^{-1/p_*}.
\end{equation}
For completeness, observe that
\[
    \nabla G_0(x)
    =
    -\frac{d-p}{p-1}
    |x-e|^{-\frac{d-p}{p-1}-2}(x-e),
\]
so that
\[
    \int_{\Sigma}
    |\nabla G_0|^p w\, \mathrm{d} x
    =
    \left(
        \frac{d-p}{p-1}
    \right)^p I.
\]
Equivalently,
\begin{equation}\label{eq:i-grad-g}
    I
    =
    \left(
        \frac{p-1}{d-p}
    \right)^p
    \|\nabla G_0\|_{L^p(\Sigma,w)}^p.
\end{equation}
It remains to pass to the homogeneous Sobolev space. Let
$u\in\dot W^{1,p}(\Sigma,w)$. By definition, there exists a sequence
$\{u_j\}\subset C^\infty_c(\overline\Sigma)$ such that
\[
    \nabla u_j
    \longrightarrow
    \nabla u
    \qquad
    \text{in }L^p(\Sigma,w).
\]
Applying \eqref{FINAL-TRACE-ST} to $u_j-u_k$, we find
\[
    \|u_j-u_k\|_{L^{p_*}(\partial\Sigma,w)}
    \leq
    C_{MT}
    \|\nabla u_j-\nabla u_k\|_{L^p(\Sigma,w)}.
\]
Thus $\{u_j|_{\partial\Sigma}\}$ is a Cauchy sequence in
$L^{p_*}(\partial\Sigma,w \dd \sigma)$. Its limit defines
$\operatorname{Tr}u$, independently of the approximating sequence, and
\eqref{FINAL-TRACE-ST} passes to the limit. Hence
\[
    \operatorname{Tr}:
    \dot W^{1,p}(\Sigma,w)
    \longrightarrow
    L^{p_*}(\partial\Sigma,w \dd \sigma)
\]
is a well-defined continuous linear operator satisfying
\eqref{eq:trace-continuous}. This completes the proof.

\subsection{Proof of Theorem~\ref{thm:sharp-trace}}

The integrability argument used in the
proof of Theorem~\ref{thm:trace-continuous} applies with
$\delta=\operatorname{dist}(e,\Sigma)>0$. On bounded subsets,
$|x-e|\geq\delta$; for $|x|\geq2|e|$, one has
$|x-e|\geq|x|/2$. Homogeneity and $(d-1)p'>d$ therefore give
$0<I_{\Sigma,w,e}<\infty$.

Using Proposition~\ref{prop:algebra-transport-map} with $c=c_{\Sigma,e}$, and repeating the optimization in the proof of Theorem~\ref{thm:trace-continuous}, we obtain the trace estimate~\eqref{eq:sharp-trace} in the transport-compatible regime.

We have to confirm that $C_\sharp$ is indeed sharp and attained.
Using the same notation, set
\[
\kappa:=\frac{d-p}{p-1}, \qquad G_0(x):=|x-e|^{-\kappa}.
\]
Direct differentiation gives
\begin{align*}
 \nabla G_0&=-\kappa|x-e|^{-\kappa-2}(x-e),\\
 |\nabla G_0|^{p-2}\nabla G_0
 &=-\kappa^{p-1}|x-e|^{-d}(x-e),\\
 \diver(w|\nabla G_0|^{p-2}\nabla G_0)
 &=-\kappa^{p-1}|x-e|^{-d}\frac{\partial w}{\partial x_N}=0.
\end{align*}
The last computation holds weakly, since $w$ is locally Lipschitz.
At $w\,\dd \sigma$-almost every regular boundary point,
\[
x\cdot\eta(x)=0,\qquad e \cdot \eta(x)=c_{\Sigma,e},
\]
and $\kappa(p_*-1)=d$, so the weighted flux identity is
\[
 w|\nabla G_0|^{p-2}\nabla G_0\cdot\eta
 =c_{\Sigma,e}\kappa^{p-1}wG_0^{p_*-1}.
\]
Choose a radial, nonincreasing function
$\chi\in C_c^\infty(\R^N)$ as in the proof of Proposition~\ref{prop:bounded-general}.

For small $\varepsilon>0$, consider 
\[
\Sigma_{\varepsilon,2R}:=\Sigma \cap (B_{2R} \setminus B_{\varepsilon}), \qquad E_{R} := \Sigma \cap (B_{2R} \setminus B_R).
\]
Let 
\[
\mathcal{V} = w|\nabla G_0|^{p-2}\nabla G_0 = -\kappa^{p-1}w(x)|x-e|^{-d}(x-e).
\]
Then,

\[
0 = \int_{\Sigma_{\varepsilon,2R}} \chi_R G_0 \diver \mathcal{V}\,\dd x = \int_{\Sigma_{\varepsilon,2R}} \diver (\chi_R G_0  \mathcal{V})\,\dd x -\int_{\Sigma_{\varepsilon,2R}} \nabla (\chi_R G_0) \cdot \mathcal{V}\,\dd x.
\]
Expanding $\nabla (\chi_R G_0)$ and using the divergence theorem,
\[
\int_{\partial \Sigma_{\varepsilon,2R}} \chi_R G_0  \mathcal{V} \cdot \eta\,\dd \sigma = \int_{\Sigma_{\varepsilon,2R}} \chi_R \nabla G_0  \cdot \mathcal{V}\,\dd x+\int_{\Sigma_{\varepsilon,2R}} G_0 \nabla \chi_R  \cdot \mathcal{V}\,\dd x.
\]
Note that $\nabla G_0 \cdot \mathcal{V} = w|\nabla G_0|^p$. Thus
\begin{equation}\label{eq:int-cut-step}
 \int_{\partial \Sigma_{\varepsilon,2R}} \chi_R G_0  \mathcal{V} \cdot \eta\,\dd \sigma = \int_{\Sigma_{\varepsilon,2R}} \chi_R |\nabla G_0|^p \, w\dd x+\int_{E_R}  G_0 |\nabla G_0|^{p-2} \nabla G_0  \cdot \nabla \chi_R\, w\dd x.  
\end{equation}
Because 
\[ 
\delta=\operatorname{dist}(e,\overline\Sigma)>0,
\]
the singularity of $G_0$ and $\nabla G_0$ are separated from $\overline \Sigma$. 

Since
\[
\partial \Sigma_{\varepsilon,2R} = (\partial \Sigma \cap (B_{2R} \setminus B_{\varepsilon})) \cup (\partial B_{2R} \cap \Sigma) \cup (\partial B_\varepsilon \cap \Sigma),
\]
the boundary integral of the left-hand side of equation~\eqref{eq:int-cut-step} can be decomposed into the sum of three integrals
\[
\int_{\partial \Sigma_{\varepsilon,2R}} \chi_R G_0  \mathcal{V} \cdot \eta\,\dd \sigma = H_{\Sigma} + H_{2R} + H_{\varepsilon},
\]
where
\begin{flalign*}
H_\Sigma &:= \int_{\partial \Sigma \cap (B_{2R} \setminus B_{\varepsilon})} \chi_R G_0  \mathcal{V} \cdot \eta\,\dd \sigma, &&\\
H_{2R} &:= \int_{\partial B_{2R} \cap \Sigma} \chi_R G_0  \mathcal{V} \cdot \left(\frac{x}{2R}\right)\,\dd \sigma, &&\\
H_\varepsilon &:= \int_{\partial B_\varepsilon \cap \Sigma} \chi_R G_0  \mathcal{V} \cdot \left(-\frac{x}{\varepsilon}\right)\,\dd \sigma.
\end{flalign*}
More precisely, this decomposition holds up to a set of $\mathcal{H}^{N-1}$-measure zero, since the intersections between the conical and spherical boundary pieces are $(N-2)$-dimensional. On the conical boundary, the condition $x\cdot\eta=0$ and transport compatibility give
\[
 \mathcal{V}\cdot\eta
 =c_{\Sigma,e}\kappa^{p-1}w(x)|x-e|^{-d},\qquad
 G_0\mathcal{V}\cdot\eta=c_{\Sigma,e}\kappa^{p-1}wG_0^{p_*}.
\]
Then,
\[
H_\Sigma = c_{\Sigma,e}\kappa^{p-1}\int_{\partial \Sigma \cap(B_{2R} \setminus B_{\varepsilon})} \chi_R G^{p_*}_0  w \dd \sigma
\]
On the boundary $\partial B_{2R} \cap\Sigma$ we have $\chi_{R} =0$ and, therefore, $H_{2R} = 0$. 
On the boundary $\partial B_{\varepsilon} \cap \Sigma$, since
$\operatorname{dist}(e, \overline\Sigma)>0$, both $G_0$ and $\mathcal{V}/w$ are
bounded near the vertex. 
Moreover, homogeneity and continuity of $w$ on the spherical section imply 
\[
    w(x)\leq C|x|^a,
\]
for every $x \in \Sigma$. Therefore
\[
 |H_\varepsilon|
 \leq C\int_{\partial B_\varepsilon \cap\Sigma}w\,\dd \sigma
 \leq C\varepsilon^{N-1+a}
 =C\varepsilon^{d-1}\longrightarrow0.
\]

Equation~\eqref{eq:int-cut-step} becomes
\[
c_{\Sigma,e}\kappa^{p-1}\int_{\partial \Sigma \cap B_{2R}} \chi_R G^{p_*}_0  w \dd \sigma=\int_{\Sigma_{2R}} \chi_R |\nabla G_0|^p \, w\dd x+\int_{E_R}  G_0 |\nabla G_0|^{p-2} \nabla G_0  \cdot \nabla \chi_R\, w\dd x.
\]
Note that, for $|x| \ge R > 2|e|$,
\[
|x-e| \ge \frac{|x|}2 \implies |x-e|^{-1} \le 2 R^{-1}.
\]
Therefore,
\[
\left| \int_{E_R}  G_0 |\nabla G_0|^{p-2} \nabla G_0  \cdot \nabla \chi_R\, w\dd x \right| \le C R^{-\kappa-d+1}R^{-1}R^d = C R^{-\kappa} \to 0.
\]
as $R \to \infty$. 
It remains to verify the convergence in the energy norm. Since
\[
\nabla(\chi_R G_0)-\nabla G_0
=
\nabla\chi_R\,G_0+(\chi_R-1)\nabla G_0,
\]
we have
\[
\begin{aligned}
\int_\Sigma
\left|\nabla(\chi_R G_0)-\nabla G_0\right|^p w\,\mathrm{d}x
&\leq
C\int_{E_R}|\nabla\chi_R|^p|G_0|^p w\,\mathrm{d}x
\\
&\quad+
C\int_{\Sigma\setminus B_R}
|\nabla G_0|^p w\,\mathrm{d}x.
\end{aligned}
\]
The first term is bounded by $CR^{-\kappa}$, while the second
tends to zero since
\[
\int_\Sigma|\nabla G_0|^p w\,\mathrm{d}x<\infty.
\]
Hence
\[
\left\|\nabla(\chi_R G_0)-\nabla G_0
\right\|_{L^p(\Sigma,w)}
\longrightarrow0
\qquad\text{as }R\to\infty.
\]

For each fixed $R$, the function $\chi_R G_0$ is smooth on $\Sigma$
and has compact support in $B_{2R}$. Moreover, the pole of $G_0$ lies at a positive distance from $\overline{\Sigma}$.
Thus $\chi_R G_0$ can be modified in a neighborhood of the pole,
disjoint from $\overline{\Sigma}$, to obtain a smooth compactly
supported function on $\R^N$ without changing its restriction to
$\Sigma$. Consequently, $\chi_R G_0$ is admissible and the convergence above proves that
$G_0\in\dot W^{1,p}(\Sigma,w)$.

Monotone convergence yields
\[
c_{\Sigma,e}\kappa^{p-1}\int_{\partial \Sigma} G^{p_*}_0  w \dd \sigma=\int_{\Sigma}  |\nabla G_0|^p \, w\dd x.
\]
Set
\[
 E_G=\int_\Sigma|\nabla G_0|^pw\dd x,\qquad
 B_G=\int_{\partial\Sigma}G_0^{p_*}w\dd\sigma.
\]
Then
\begin{equation}\label{eq:profileidentities}
 E_G=\kappa^pI_{\Sigma,w,e},\qquad
 E_G=c_{\Sigma,e}\kappa^{p-1}B_G,\qquad
 B_G=\frac{\kappa I_{\Sigma,w,e}}{c_{\Sigma,e}}.
\end{equation}
Thus $G_0\in\dot W^{1,p}(\Sigma,w)$ and its abstract trace is its
boundary restriction, as follows also from boundary dominated
convergence. From~\eqref{eq:profileidentities},
\[
 \frac{B_G^{1/p_*}}{E_G^{1/p}}
 =\bigl[c_{\Sigma,e}\kappa^{p_*-1}
 I_{\Sigma,w,e}^{\frac{p-1}{d-p}}\bigr]^{-1/p_*}.
\]
This proves attainment and sharpness. The sharp constant is then obtained by
\[
C_\sharp
=
\left[
c_{\Sigma,e}
\left(\frac{d-p}{p-1}\right)^{p_*-1}
I_{\Sigma,w,e}^{\frac{p-1}{d-p}}
\right]^{-1/p_*}.
\]
It remains to verify equality for the family~\eqref{eq:explicit-extremals}.
For $\lambda>0$, define
\[
u_{C,\lambda}(x)
=
C\lambda^{\frac{d-p}{p}}
|\lambda x-e|^{-\frac{d-p}{p-1}},
\qquad C\neq0.
\]
Since
\[
G_0(\lambda x)
=
|\lambda x-e|^{-\frac{d-p}{p-1}},
\]
we may write
\[
u_{C,\lambda}(x)
=
C\lambda^{\frac{d-p}{p}}G_0(\lambda x).
\]
The homogeneity of $w$ gives
\[
w(\lambda x)=\lambda^a w(x),
\]
and hence, after the change of variables $y=\lambda x$,
\[
\int_{\Sigma}
|\nabla u_{C,\lambda}|^p w\,\mathrm{d}x
=
|C|^p
\int_{\Sigma}
|\nabla G_0|^p w\,\mathrm{d}x.
\]
Similarly, using
\[
p_*\frac{d-p}{p}=d-1
\]
and the homogeneity of the boundary measure,
\[
w(x)\,\mathrm{d}\sigma(x)
=
\lambda^{-(N-1+a)}
w(y)\,\mathrm{d}\sigma(y)
=
\lambda^{-(d-1)}
w(y)\,\mathrm{d}\sigma(y),
\]
we obtain
\[
\int_{\partial\Sigma}
|u_{C,\lambda}|^{p_*}w\,\mathrm{d}\sigma
=
|C|^{p_*}
\int_{\partial\Sigma}
G_0^{p_*}w\,\mathrm{d}\sigma.
\]
Therefore inequality~\eqref{eq:sharp-trace} is invariant under multiplication by
$C$ and under the dilation parameter $\lambda$. Since equality holds for the normalized profile $c_0G_0$, it follows that equality holds for every
\[
u_{C,\lambda}(x)
=
C\lambda^{\frac{d-p}{p}}
|\lambda x-e|^{-\frac{d-p}{p-1}},
\qquad C\neq0,\quad \lambda>0.
\]

Hence $C_\sharp$ is optimal and is attained by every function of the
form~\eqref{eq:explicit-extremals}.

\subsection{Proof of Theorem~\ref{thm:classification}}

Let $u \in \dot W^{1,p}(\Sigma,w)$ be a nonzero extremal. Its trace norm is positive; otherwise
equality would force zero energy and hence $u\equiv0$. The function $|u|$
has the same energy and trace norm, so initially consider a
nonnegative extremal. Variations $u+t\zeta$ with
$\zeta\in C_c^\infty(\Sigma)$ preserve its boundary trace. Extremality
therefore implies
\begin{equation}\label{eq:euler-lagrange-extrema}
 \int_\Sigma |\nabla u|^{p-2}\nabla u\cdot\nabla\zeta\,w\dd x=0.
\end{equation}
The differentiation is justified by $p>1$ and the usual $p$-growth bound as in~\eqref{eq:p-growth-bound}.
From~\eqref{eq:local-positive-lower-bound} in Section~\ref{sec:preliminaries} we conclude
\[
 u>0\text{ in }\Sigma,\qquad \inf_Ku>0,\qquad \text{ for every compact }K\subset\Sigma
\]
At this point, normalize $u$ by setting $\|u\|_{L^{p^*}(\Sigma,w)} = 1$. Consider $v = c_0G_0$ as in definition~\eqref{eq:def-G0} and transport the probability measures $\dd \mu=u^{p^*}w\dd x$ onto $\dd \nu=v^{p^*}w\dd x$.  Let $\varphi$ be the convex transport potential defined
as in Lemma~\ref{lem:trucation-potential} and Theorem~\ref{thm:general}.

To retain the contribution of the singular part of the
distributional Hessian in the equality analysis, we use
a local approximation argument related to that of
Balogh, Don, and Krist\'aly~\cite[Section~3.2]{BDK2024}.
In the present trace setting, the approximating sequence
must also preserve convergence of the boundary traces.
We therefore modify the approximants only on compact
subsets of $\Sigma$, obtaining local uniform convergence
while leaving their boundary traces unchanged.

Let
$\mu_s := \operatorname{tr}[D_s^2\varphi] \geq 0$.
Fix a compact set $K\subset\Sigma$. Since $\Sigma$ is open, $ \operatorname{dist}(K,\partial\Sigma)>0$.
Choose $\zeta\in C_c^\infty(\Sigma)$ such that
\[
    0\leq\zeta\leq1,
    \qquad
    \zeta\equiv1
    \quad\text{in a neighborhood of }K.
\]
Consider the chain of inclusions:
\[
K \subset \subset \{\zeta=1\}\subset \subset \supp \zeta \subset \subset V \subset \subset \Sigma.
\]
Let $(u_j)$ be the approximating sequence as in equation~\eqref{eq:approximation}.
Consider the standard mollifier
\[
\rho_{\varepsilon}(x) = \varepsilon^{-N}\rho\left(\frac{x}{\varepsilon}\right), \qquad \rho \in C^{\infty}_c(B_1), \qquad \rho \ge 0, \qquad \int_{\R^N} \rho \dd x =1.
\]
Consider the decreasing sequence $(\varepsilon_j)_{j \in \mathds{N}}$  converging to $0$, such that
\[
0<\varepsilon_j<\frac12\dist(\supp \zeta,\R^N \setminus V),\qquad \forall j \ge 1.
\]
Choose $\theta \in C^{\infty}_c(\Sigma)$ with $\theta \equiv 1$ on a neighborhood of $\overline V$ and extend $\theta u$ by zero to $\R^N$.
Define
\[
 \widetilde u_j
 =\zeta((\theta u)*\rho_{\varepsilon_j})+(1-\zeta)u_j.
\]
The mollification is performed in a larger interior neighborhood of $K$.
There the weight is bounded above and below, so convolution converges
in the local Sobolev norm and, by continuity of $u$, uniformly on $K$.
The terms involving $\nabla\zeta$ converge by local $L^p$ convergence.
Consequently $\widetilde u_j$ still converges globally in the energy
and interior norms, agrees with $u_j$ near the boundary, and converges
uniformly on $K$. A diagonal construction works on a compact exhaustion.
These functions are nonnegative and admit compactly supported $C^1$ extensions to $\R^N$.

Fix $k$ and define smooth cut-off functions $\chi_R$ as in Proposition~\ref{prop:bounded-general}. Take $R$ with $\chi_R\equiv1$ on $K$. Using equation~\eqref{eq:basic}, applied
to $\varphi_k$ and $\chi_R\widetilde u_j^{p_*}$, 
we obtain
\begin{align*}
 &\int_\Sigma
 \chi_R\widetilde u_j^{p_*}
 \bigl(w\Delta_A\varphi_k+\nabla w\cdot\nabla\varphi_k\bigr)\dd x
 +\int_\Sigma
 \chi_R\widetilde u_j^{p_*}
 \, w\dd \mu_{s,k}
 +
 c_{\Sigma,e}
 \int_{\partial\Sigma}\chi_R(\Tr \widetilde u_j)^{p_*}w\dd\sigma
 \\
 &\quad
 \leq
 -p_*\int_\Sigma
 \chi_R\widetilde u_j^{p_*-1}\nabla \widetilde u_j\cdot(\nabla\varphi_k-e)w\dd x- \int_\Sigma
 \widetilde u_j^{p_*}\nabla \chi_R\cdot(\nabla\varphi_k-e)w\dd x
\end{align*}
where $\mu_{s,k}:=\tr[D^2_s \varphi_{k}]$. The following integral inequality
\[
\int_\Sigma
 \chi_R\widetilde u_j^{p_*}
 \, w\dd \mu_{s,k} \ge \int_K
 \widetilde u_j^{p_*}
 \, w\dd \mu_{s,k}
\]
holds since the complement
\[
\int_{\Sigma \setminus K}
 \chi_R\widetilde u_j^{p_*}
 \, w\dd \mu_{s,k} \ge 0.
\]
Therefore
\begin{align*}
 &\int_\Sigma
 \chi_R\widetilde u_j^{p_*}
 \bigl(w\Delta_A\varphi_k+\nabla w\cdot\nabla\varphi_k\bigr)\dd x
 +\int_K
 \widetilde u_j^{p_*}
 \, w\dd \mu_{s,k}
 +
 c_{\Sigma,e}
 \int_{\partial\Sigma}\chi_R(\Tr \widetilde u_j)^{p_*}w\dd\sigma
 \\
 &\quad
 \leq
 -p_*\int_\Sigma
 \chi_R\widetilde u_j^{p_*-1}\nabla \widetilde u_j\cdot(\nabla\varphi_k-e)w\dd x- \int_\Sigma
 \widetilde u_j^{p_*}\nabla \chi_R\cdot(\nabla\varphi_k-e)w\dd x
\end{align*}

To pass to the limits $j\to\infty$ and $R\to\infty$, we proceed
as in the proof of Proposition~\ref{prop:bounded-general}.
Fix $k$ and a compact set $K\subset\Sigma$.
Uniform convergence on $K$ and the finiteness of $\mu_{s,k}(K)$ imply
\[
 \int_K \widetilde u_j^{p_*} w\,\dd\mu_{s,k}
 \longrightarrow\int_K u^{p_*} w\,\dd\mu_{s,k}.
\]

Before passing $k \to \infty$, in order to apply the assertion in~\eqref{local-equality}, we must verify that
\[
\Sigma \subset \mathrm{int}(  \mathrm{dom}\, \varphi).
\]
Set
$D=\operatorname{dom}\varphi$. Positivity of the source density gives
\[
 \mathcal L^N(\Sigma\setminus D)=0.
\]
Consequently $\Sigma\subset\overline D$: otherwise a small ball in
$\Sigma$ disjoint from $\overline D$ would have positive measure.
Because $\Sigma$ is open,
\[
 \Sigma\subset\operatorname{int}(\overline D).
\]
The convex set $D$ is full-dimensional, and the elementary convex-set
identity
\[
 \operatorname{int}(\overline D)=\operatorname{int}D
\]
gives the desired conclusion.

To pass to the limit $k \to \infty$, we proceed as in the proof of Theorem~\ref{thm:general}. By~\eqref{local-equality}, the singular Hessian measures agree on $K$, for $k$ sufficiently large, so that $\mu_{s,k} |_K=\mu_s |_K$.
Proceeding with exhaustion of $\Sigma$ by compact
sets and using monotone convergence,
\begin{align}\label{singularineq}
 &\int_\Sigma u^{p_*}(w\Delta_A\varphi+\nabla w\cdot \nabla \varphi)\dd x
 +\int_\Sigma u^{p_*}w\dd\mu_s
 +c_{\Sigma,e}\int_{\partial\Sigma}(\Tr u)^{p_*}w\dd\sigma\notag\\
 &\hspace{10mm}\leq
 -p_*\int_\Sigma u^{p_*-1}\nabla u\cdot(\nabla \varphi-e)w\dd x.
\end{align}

We recall the chain of inequalities using the same procedure by transporting $\mu$ onto $\nu$ as in Proposition~\ref{prop:algebra-transport-map} and using the notation and formulas obtained in the proofs of Theorem~\ref{thm:trace-continuous} and Theorem~\ref{thm:sharp-trace},
\[
dc^{p_*}_0I=d\int_{\Sigma} v^{p_*}\, w\,\dd x = d\int_{\Sigma} u^{p_*} \left ( \frac {w(\nabla \varphi)}{w} \right )^{1/d} (\det D^2_A \varphi)^{1/d}\, w\,\dd x.
\]

The following estimate combines the weight inequality~\eqref{eq:weight-conc-st}, the matrix arithmetic–geometric mean inequality for $\det D^2_A \varphi$, and the scalar weighted arithmetic–geometric mean inequality with weights $a/d$ and $N/d$,
\[
d\int_{\Sigma} u^{p_*} \left ( \frac {w(\nabla \varphi)}{w} \right )^{1/d} (\det D^2_A \varphi)^{1/d}\, w\,\dd x \le \int_\Sigma u^{p_*}(w\Delta_A\varphi+\nabla w\cdot \nabla \varphi)\dd x.
\]
The next inequality adds the nonnegative singular term $\int_{\Sigma}u^{p_*}w\dd \mu_s$,
\[
 \int_\Sigma u^{p_*}(w\Delta_A\varphi+\nabla w\cdot \nabla \varphi)\dd x \le \int_\Sigma u^{p_*}(w\Delta_A\varphi+\nabla w\cdot \nabla \varphi)\dd x
 +\int_\Sigma u^{p_*}w\dd\mu_s
\]
At the third step, we use inequality~\eqref{singularineq}:
\begin{align*}
    \int_\Sigma u^{p_*}(w\Delta_A\varphi+\nabla w\cdot \nabla \varphi)\dd x
 +\int_\Sigma u^{p_*}w\dd\mu_s
    \le &\\
    -p_*\int_\Sigma u^{p_*-1}\nabla u\cdot(\nabla \varphi-e)w\dd x - c_{\Sigma,e}\int_{\partial\Sigma}(\Tr u)^{p_*}w\dd\sigma.
\end{align*}
The last step uses H\"older's inequality and the push-forward identity to obtain
\begin{align*}
 \left |-p_*\int_\Sigma u^{p_*-1}\nabla u\cdot(\nabla\varphi-e)w\,\dd x \right |
 &\leq p_*\mathcal{K}(u)
 \left(\int_\Sigma u^{p^*}|\nabla\varphi-e|^{p'}w\,\dd x\right)^{1/p'}\\
 &=p_*\mathcal{K}(u)
 \left(\int_\Sigma v^{p^*}|y-e|^{p'}w(y)\,\dd y\right)^{1/p'}\\
 &=p_*\mathcal{K}(u)(c_0^{p^*} I)^{1/p'}.
\end{align*}

where $\mathcal{K}(u)=\|\nabla u\|_{L^{p}(\Sigma,w)}/\|u\|_{L^{p^*}(\Sigma,w)} = \|\nabla u\|_{L^{p}(\Sigma,w)}$, for the normalized $u$. Since $u$ is extremal
\[
\int_{\partial\Sigma}(\Tr u)^{p_*}w\dd\sigma = C^{{p_*}}_{\sharp} \mathcal{K}(u)^{p_{*}}= c^{-1}_{\Sigma,e}\kappa^{1-p_*}I^{\frac{1-p}{d-p}}\mathcal{K}(u)^{p_{*}},
\]
where $I=I_{\Sigma,w,e}$. We reach the following inequality
\begin{equation}\label{eq:check-final}
dc^{p_*}_0I + \frac{\mathcal{K}(u)^{p_{*}}}{\kappa^{p_*-1}I^{\frac{1}{\kappa}}}  \le p_*\mathcal{K}(u)c_0^{p^*/p'}I^{1/p'}.
\end{equation}
To clarify the chain of integral inequalities define:

\begin{flalign*}
I_1 &:= d\int_{\Sigma} v^{p_*}\, w\,\dd x, &&\\
I_2 &:= d\int_{\Sigma} u^{p_*} \left ( \frac {w(\nabla \varphi)}{w} \right )^{1/d} (\det D^2_A \varphi)^{1/d}\,w\,\dd x, &&\\
I_3 &:= \int_\Sigma u^{p_*}(w\Delta_A\varphi+\nabla w\cdot \nabla \varphi)\dd x, &&\\
I_4 &:= \int_\Sigma u^{p_*}(w\Delta_A\varphi+\nabla w\cdot \nabla \varphi)\dd x
    +\int_\Sigma u^{p_*}w\dd\mu_s, &&\\
I_5 &:= p_*\int_\Sigma u^{p_*-1}\nabla u\cdot(\nabla \varphi-e)w\dd x, &&\\
I_6 &:= c_{\Sigma,e}\int_{\partial\Sigma}(\Tr u)^{p_*}w\dd\sigma. &&
\end{flalign*}
We have established
\[
I_1 = I_2 \le I_3 \le I_4 \le -I_5 -I_6,
\]
or equivalently
\[
I_1 +I_6= I_2 +I_6\le I_3 +I_6\le I_4+I_6 \le -I_5.
\]
Then equation~\eqref{eq:check-final} is, indeed,
\[
dc^{p_*}_0I + \frac{\mathcal{K}(u)^{p_{*}}}{\kappa^{p_*-1}I^{\frac{1}{\kappa}}}=I_1+I_6 \le -I_5 \le |I_5| \le  p_*\mathcal{K}(u)(c_0^{p^*}I_{})^{1/p'},
\]
where the last inequality was obtained by  H\"older's inequality and the push-forward identity. Now, for $c=c_{\Sigma,e}$, consider
\[
    A_c
    :=
    \frac{p_*}{c}
    c_0^{p^*/p'}I^{1/p'},
    \qquad
    B_c
    :=
    \frac{d}{c}
    c_0^{p_*}I,\qquad \Phi_c(t):=\frac{t^{p_*}}{A_ct-B_c},
    \qquad
    t>\frac{B_c}{A_c}.
\]
We have from the proof of Theorem~\ref{thm:trace-continuous},
\[
\mathcal{J}(u)^{p_*} \ge \Phi_c(\mathcal{K}(u))\ge \min \Phi_c.
\]
For the normalized extremal $u$ we must achieve the minimum, 
\[
\mathcal{J}(u)^{p_*} = C^{-p_*}_{\sharp}=\min_{t>B_c/A_c} \Phi_{c}(t).
\]
Therefore
\[
\mathcal{K}(u) = t_0 =
    \kappa
    c_0 I^{1/p}.
\]
If we use the value $\mathcal{K}(u)=\kappa c_0 I^{1/p}$ in equation~\eqref{eq:check-final} we have the equality
\[
I_1+I_6 = p_*\mathcal{K}(u)(c_0^{p^*}I_{})^{1/p'}.
\]
Since the endpoints of the chain coincide, equality holds at every
intermediate step. For instance,
\[
I_4 - I_3 = 0 \implies \int_\Sigma u^{p_*}w\dd\mu_s=0.
\]
Then for every compact $K\subset\Sigma$,
\[
 0=\int_\Sigma u^{p_*}w\dd\mu_s
 \geq(\inf_Ku)^{p_*}(\min_Kw)\mu_{s}(K),
\]
so $\mu_{s}=0$. A positive semidefinite matrix measure with zero trace is
zero; therefore $D_s^2\varphi=0$ in $\Sigma$ and $D^2 \varphi = D^2_A \varphi \dd x$.
Additionally,
\[
I_3 - I_2 =0 \implies \tr D_A^2 \varphi -N(\det D_A^2 \varphi )^{1/N}=0.
\]
We conclude that all eigenvalues of $D^2 \varphi$ are equal. The Jacobian identity and positivity
of source and target densities give $\det D^2\varphi>0$ almost everywhere. Thus
\begin{equation}\label{scalarhessian}
 D^2\varphi=\lambda(x)I_N\dd x,\qquad
 \lambda(x)>0\text{ a.e.}
\end{equation}
For each $j$, choose
$i\ne j$. Equation~\eqref{scalarhessian} gives
\[
 \partial_j\lambda=\partial_j\partial_{ii}\varphi
 =\partial_i\partial_{ij}\varphi=0.
\]
Connectedness implies that $\lambda$ is constant, and hence
\begin{equation}\label{affine}
 \nabla \varphi(x)=\lambda x+b,\qquad \lambda>0.
\end{equation}
Let
\[
A(x):=\lambda x+b,\qquad \lambda>0.
\]
Since the source and target measures have support $\overline \Sigma$, and $A$ is a homeomorphism pushing one
measure onto the other,
\[
A(\overline \Sigma) = \overline \Sigma.
\]

Consequently,
\[
 \lambda \overline \Sigma + b=\overline \Sigma.
\]
Because $\lambda \overline \Sigma = \overline \Sigma$, one gets
\[
 \overline \Sigma+b=\overline \Sigma.
\]
Then $0\in \overline \Sigma$ gives $b\in \overline \Sigma$
and $-b\in \overline \Sigma$. Hence $b\in \operatorname{L}(\Sigma)$, and
\[
 x_0=-\frac b\lambda\in \operatorname{L}(\Sigma).
\]
The translations in $\operatorname{L}(\Sigma)$ preserve $\Sigma$. The lineality invariance established
before Theorem~\ref{thm:classification} gives 
\[
 w(x+v)=w(x)
 \qquad\text{for }v\in \operatorname{L}(\Sigma).
\]
Therefore,
\[
 w\bigl(\lambda(x-x_0)\bigr)=\lambda^a w(x).
\]
The affine change of the variables $x \mapsto \lambda(x-x_0)$, with $x_0 = -b/\lambda$, gives
\[
 u(x)^{p^*}w(x)
 =\lambda^N v(\lambda(x-x_0))^{p^*}w(\lambda(x-x_0))
 =\lambda^d v(\lambda(x-x_0))^{p^*}w(x).
\]
After canceling $w(x)>0$ and using $d/p^*=(d-p)/p$,
\[
 u(x)=c_0\lambda^{(d-p)/p}
       |\lambda(x-x_0)-e|^{-(d-p)/(p-1)}.
\]
Removing amplitude normalization gives the positive family of extremals.

For an original signed extremal, apply the preceding argument to $|u|$.
Since $|u|$ is bounded away from zero on each compact subset of $\Sigma$,
the sign function admits a Lipschitz extension on the local range of $u$.
The Sobolev chain rule then gives
\[
\nabla\!\left(\frac{u}{|u|}\right)=0
\qquad\text{locally in }\Sigma.
\]
Since $\Sigma$ is connected, $u/|u|$ is constant, and hence $u$ has
a constant sign.

Connectedness makes it constant. Thus every
nonzero extremal has the form
\begin{equation}\label{allprofiles}
 u(x)=C\lambda^{(d-p)/p}
       |\lambda(x-x_0)-e|^{-(d-p)/(p-1)},\qquad
 C\ne0,\quad\lambda>0,\quad x_0\in \operatorname{L}(\Sigma).
\end{equation}
Conversely, translations along $\operatorname{L}(\Sigma)$ preserve
both measures. Dilation gives the energy factor
$\lambda^{d-N-a}$ and the boundary factor
$\lambda^{d-1-(N-1)-a}$, both equal $1$. Amplitudes multiply both norms by $|C|$.
The attainment proved in Theorem~\ref{thm:sharp-trace} therefore proves
that every function in \eqref{allprofiles} is an extremal.

\section*{Statements and Declarations}

\noindent\textbf{Funding.}
This work was financed in part by the Coordena\c{c}\~ao de
Aperfei\c{c}oamento de Pessoal de N\'ivel Superior -- Brasil (CAPES) --
Finance Code 001, and by the MATH-AmSud grant 88887.878894/2023-00.
J. M. do \'{O} also acknowledges partial support from CNPq through grants
312435/2026-6, 409764/2023-0, and 443594/2023-6, and from the Para\'iba
State Research Foundation (FAPESQ), grant no.~3034/2021.

\medskip

\noindent\textbf{Conflict of interest.}
The authors declare that they have no conflicts of interest.

\medskip

\noindent\textbf{Authors' contributions.}
Both authors contributed to the conception of the work, the development
of the mathematical arguments, and the writing and revision of the
manuscript. Both authors read and approved the final version of the
manuscript.

\medskip

\noindent\textbf{Data availability.}
Data sharing is not applicable to this article, as no datasets were
generated or analyzed during the study.

\bigskip


\bibliographystyle{amsplain}
\bibliography{biblio}

\end{document}